\documentclass[12pt]{amsart}
\usepackage{amsthm,amssymb,amsmath,amscd,epic,eepic}
\usepackage[all]{xy}

\newcommand     {\comment}[1]   {}
 \newcommand{\mute}[2] {}
\newcommand     {\printname}[1] {}
\newcommand{\labell}[1] {\label{#1}}

\numberwithin{equation}{section}
\newtheorem {Theorem}                   {Theorem}
\newtheorem*{Theorem*}                   {Theorem}
\newtheorem {refTheorem}[equation]      {Theorem}
\newtheorem {Lemma}[equation]           {Lemma}
\newtheorem* {Lemma*}                    {Lemma}

\newtheorem {Corollary} [equation]      {Corollary}
\newtheorem* {Corollary*}                {Corollary}
\newtheorem {Proposition} [equation]    {Proposition}

\theoremstyle{definition}
\newtheorem{Definition}[equation]{Definition}
\newtheorem*{Definition*}{Definition}

\theoremstyle{remark}
\newtheorem{Remark}[equation]{Remark}
\newtheorem*{Remark*}{Remark}

\newtheorem*{Example*}{Example}

\long\def\symbolfootnote[#1]#2{\begingroup%
\def\thefootnote{\fnsymbol{footnote}}\footnote[#1]{#2}\endgroup} 

\def    \C             {{\mathbb C}}
\def    \R             {{\mathbb R}}

\def    \cO             {{\mathcal O}}

\def    \cE             {{\mathcal E}}
\def    \cL             {{\mathcal L}}

\def \Hat 	{\widehat}
\def \ssetminus	{\smallsetminus}

\def \del	{{\partial}}
\def \delbar	{\overline{\partial}}

\def \Tilde	{\widetilde}

\def \P	{\mathbb P}

\def    \inv    {^{-1}}

\def    \cut  {{\operatorname{cut}}}

\def	\modt	{{/\!/}_{\! t}\,  S^1}
\def	\modc	{{/\!/}_{\! c_+}\,  S^1}

\def	\modt	{{/\!/}_{\! t}\,  S^1}

\def    \CP	{{\mathbb C}{\mathbb P}}
\def    \R	{{\mathbb R}}

\def    \C	{{\mathbb C}}
\def    \Z       {{\mathbb Z}}

\def    \U      {{\operatorname{U}}}

\def    \tomega       {{\widetilde{\omega}}}
\def    \tepsilon       {{\widetilde{\epsilon}}}
\def    \tPsi       {{\widetilde{\Psi}}}
\def    \tM       {{\widetilde{M}}}

\def    \tJ       {{\widetilde{J}}}

\def    \tsigma       {{\widetilde{\sigma}}}

\begin{document}

\title[Non-Hamiltonian actions]
{Non-Hamiltonian actions with fewer isolated fixed points}

\author[Donghoon Jang]{Donghoon Jang}
\address{Department of Mathematics, Pusan National University, Pusan, Korea}
\email{donghoonjang@pusan.ac.kr}

\author[Susan Tolman]{Susan Tolman}
\address{Department of Mathematics, University of Illinois at Urbana-Champaign,
 Urbana, IL 61801}
\email{stolman@math.uiuc.edu}

\thanks{Donghoon Jang is supported by Basic Science Research Program through the National Research Foundation of Korea(NRF) funded by the Ministry of Education(2018R1D1A1B07049511). Susan Tolman is partially supported by National Science Foundation Grant DMS-1206365, and Simons Foundation Collaboration Grant 637996.}

\begin{abstract}
In an earlier paper, the second author resolved a question of McDuff by constructing a non-Hamiltonian symplectic circle action on a closed, connected six-dimensional symplectic manifold with exactly 32 fixed points.
In this paper, we improve on this example by reducing the number of fixed points.
More concretely, we construct a non-Hamiltonian symplectic circle action 
on a closed, connected six-dimensional symplectic manifold 
with exactly $2k$ fixed points for any $k \geq 5$.
 \end{abstract}

\maketitle
\section{Introduction}

Let $(M,\omega)$ be a (non-empty) closed, connected $2n$-dimensional symplectic manifold.
Let the circle $S^1 \simeq \R/\Z$ act on $M$,
inducing a vector field\footnote{In this paper, we always take the vector field associated to a primitive lattice element in the Lie algebra of $S^1$.}
$\xi_M$ on $M$.
Assume that the action preserves the symplectic form, that is, the action is  {\bf symplectic}.
Equivalently, assume that $\iota_{\xi_M} \omega$
is closed.  
The {\bf (isotropy) weights} at a fixed point $p \in M$ are the elements of the
unique multiset of  integers $\{w_1,\dots,w_n\}$ so that the symplectic representation of $S^1$ on $T_p M$
is isomorphic to the representation on $\big(\C^n, \frac{\sqrt{-1}}{2} \sum_i d z_i \wedge d \overline{z}_i \big) $ given by 
$\lambda \cdot z = (\lambda^{w_1} z_1,\dots,\lambda^{w_n} z_n)$.

The action is {\bf Hamiltonian} if there exists a {\bf moment map}, that is, a map
$\Psi \colon M \to \R$ satisfying
$$d \Psi = - \iota_{\xi_M} \omega.$$
Equivalently, the action is Hamiltonian if  $\iota_{\xi_M} \omega$ is exact.
In this case, we can reduce the number of degrees of freedom by passing
to the {\bf reduced space} $M \modt := \Psi^{-1}(t)/S^1$, which is 
a $(2n-2)$-dimensional symplectic orbifold for all regular values $t$  of $\Psi$.
Additionally, the moment map $\Psi$ is a Morse-Bott function whose critical set is the fixed set $M^{S^1}$. The index $\mu_F$ of a fixed component
$F \subset M^{S^1}$ is twice the number of negative weights at any $p \in F$.
This implies, for example,
that $\Psi$ is a perfect Morse function, i.e.,
\begin{equation}\labell{Morse}
\dim H^i(M;\R) = \sum_{F \subset M^{S^1}} \dim H^{i - \mu_F}(F;\R),
\end{equation}
where the sum is over all fixed components \cite{Ki}.
Additionally, since $\Psi$ has a minimum, $M$ must have a fixed point with no negative weights.

This raises a natural question: Are symplectic actions always Hamiltonian?
If $H^1(M;\R) = 0$, then every symplectic action is Hamiltonian because every closed one-form is exact.
In contrast, there are many symplectic circle actions with no fixed points,
including the diagonal action on the torus $S^1 \times S^1$.
These actions are never Hamiltonian.

In 1959, Frankel made significant progress towards understanding which symplectic actions are Hamiltonian
by proving that a K\"ahler circle action on a closed, connected K\"ahler manifold is Hamiltonian exactly
if the fixed set is nonempty \cite{Fr}.
In 1988, McDuff proved that a symplectic circle action on a $4$-dimensional closed, connected symplectic manifold is Hamiltonian exactly if the  fixed set is nonempty \cite{Mc}. Since then,  many authors have proved additional results showing that symplectic actions must be Hamiltonian in a variety of circumstances \cite{CHS}, \cite{Fe}, \cite{G}, \cite{J3}, \cite{Kot}, \cite{LM}, \cite{Li}, \cite{MW}, \cite{O}, \cite{Ro}, \cite{S}, \cite{TWe}; 
see \cite{T} for a brief survey. 
On the other hand, in her 1988 paper McDuff  also constructed
a non-Hamiltonian symplectic circle action on a $6$-dimensional closed, connected symplectic manifold
with fixed tori.
However, for a long time, this was the only known example.
This led to the following more refined question, often known as the ``McDuff conjecture" \cite{MS}: 
Does there exist a non-Hamiltonian symplectic circle action
on a closed, connected symplectic manifold with a non-empty discrete fixed set?
In 2017, the second author answered this question by constructing
a non-Hamiltonian symplectic circle action on a six-dimensional
closed, connected symplectic manifold with exactly 32 isolated fixed points \cite{T}.
As a corollary, examples exist in all dimensions greater than $4$.

This relates to another important question: What is the minimum positive number
of fixed points for a symplectic action on a closed, connected $2n$-dimensional
manifold? If the action is Hamiltonian then
$[\omega^i] \in H^{2i}(M)$ is nonvanishing for all $0 \leq i \leq n$, 
and so \eqref{Morse} implies that $M$ has at least $n+1$ fixed points.
This bound is sharp; generic $S^1$ subgroups of the natural $(S^1)^n/S^1$ action on $\CP^n$ have $n +1$ fixed points.
More generally, by the Atiya-Bott-Berline-Vergne localization theorem,
any circle action that preserves an almost complex structure,
and has at least one fixed point, must have more than one fixed point. 
Additionally, Kosniowski proved that it must have more
than two fixed points if $n \notin \{1,3\}$ \cite{Kos1}, \cite{Kos2}; see also \cite{PT}.
Finally, the first author proved that it cannot have 3 fixed points unless $n=2$ \cite{J1}, \cite{J2}; see also \cite{Ku}.
Kosniowski conjectured that any circle action that preserves an almost complex structure,
and has at least one fixed point must have more than $\frac{n}{2}$ fixed points.
This bound is sharp when $n = 3$;
there's an almost complex structure on $S^6$ which is preserved by an $S^1$ action with two fixed points.
This raises another important question: Is there a (necessarily non-Hamiltonian)
symplectic circle action on a closed $6$-dimensional symplectic
manifold with exactly $2$ fixed points?


In this paper, we show that the number fixed points
in the second author's examples can be substantially reduced.

\begin{Theorem}\label{maintheorem}
Given an integer $k \geq 5$,
 there exists a non-Hamiltonian symplectic circle action on a closed, connected six-dimensional symplectic manifold with exactly $2k$ fixed points.
\end{Theorem}

In particular, there exists an example with exactly $10$ fixed points,
which we describe in more detail below.
To do so, we
need some additional terminology.

Let's again assume that the circle acts symplectically
on a closed $2n$-dimensional symplectic 
manifold $(M,\omega)$. 
If $[\omega] \in H^2(M;\R)$ is integral, there exists a 
{\bf generalized moment map}, that is, 
a map $\Psi \colon M \to S^1$ satisfying $\Psi^*(dt) = - \iota_{\xi_M} \omega$.
As in the Hamiltonian case, the {\bf reduced space} $M \modt :=
\Psi^{-1}(t)/S^1$ is a $(2n -2)$-dimensional orbifold
endowed with a symplectic form $\omega_t \in \Omega^2(M \modt)$
for all regular values $t$  of $\Psi$. 
The {\bf Duistermaat-Heckman}
function of $M$ is the unique continuous real-valued function $\varphi$ on the
moment image $\Psi(M)$ to $\R$ such that
$$\varphi(t) = \int_{M \modt} \omega_t^{n-1}$$
for all regular values $t$ of $\Psi$ \cite{DH}.

A \textbf{K3 surface} is a  closed, connected complex surface $(\Hat X, \Hat I)$
with $H^1(\Hat X;\mathbb{R})=\{0\}$ and trivial canonical bundle.
A closed complex surface $(X,I)$ with at worst simple singular points and with minimal resolution $q \colon \Hat X \to X$ is a {\bf generalized K3 surface} if $\Hat X$ is a K3 surface; see Section 4 for more details.
In this paper, we focus on generalized K3 surfaces that only
have isolated $\Z_2$ singularities. (When we 
specify the number of such singularities, 
we are implying that the surface is otherwise smooth.)
For example, the quotient of the complex torus $\mathbb{C}^2/(\mathbb{Z}^2+ \sqrt{-1} \, \mathbb{Z}^2)$ by the involution $[z] \to [-z]$ is a generalized K3 surface with 16 isolated $\mathbb{Z}_2$ singularities.
A symplectic form $\sigma \in \Omega^2(\Hat X)$ is {\bf tamed} if $\sigma(v,\Hat I(v)) > 0$ for every
nonzero tangent vector $v$.
In this case, we say that the triple $(\Hat X, \Hat I, \sigma)$ is a {\bf tame K3 surface}.
A tamed symplectic form $\sigma \in \Omega^2(\Hat X)$ is  {\bf K\"ahler} if, additionally, $\sigma(\Hat I(v),\Hat I(w)) = \sigma(v,w)$
for all tangent vectors $v$ and $w$. 

\begin{Theorem} \label{theorem10}
There exists a non-Hamiltonian symplectic circle action on a closed, connected six-dimensional symplectic manifold $(M,\omega)$ 
with a generalized moment map $\Psi \colon M \rightarrow \mathbb{R}/10\Z\simeq S^1$ satisfying the following properties: 
The level sets $\Psi^{-1}(\pm 1)$  each contain 5 fixed points with weights $\pm\{2,-1,-1\}$; otherwise, the action is locally free. The Duistermaat-Heckman function is
\begin{equation*}
\varphi(t)=\begin{cases} 
12-2t^2 & -1  \leq t \leq 1
\\ 
2 + \frac{1}{2}(t-5)^2 & \ \ 1  \leq t \leq 9.
\end{cases}
\end{equation*}
Finally, the reduced space $M\modt$ is diffeomorphic to a generalized K3 surface with $5$ isolated
$\Z_2$ singularities for all $t \in (1,9)$, and symplectomorphic to a tame K3 surface for all $t \in (-1,1)$. \end{Theorem}

We now briefly outline the construction of the symplectic manifold 
with $10$ fixed points
delineated in Theorem~\ref{theorem10}.
The ideas behind the  construction of the more general case described in Theorem~\ref{maintheorem}  are identical,
although the details are slightly more complicated.


First, in Section~\ref{free}, we review the classification of free circle actions
on symplectic manifolds\footnote{In this paper, we only consider manifolds {\em without} boundary.}
 with proper moment maps to an open interval $(a,b) \subset \R$,
under the assumption that the reduced space is symplectomorphic to a tame K3 surface $(X',I'_t,\sigma'_t)$ for all $t \in (a,b)$; we also need to insist on the technical assumption that $[\sigma'_t] = \kappa' - t \eta'$ for all $t \in (a,b)$, 
where $\kappa',\eta'$ induce
a primitive embedding\footnote{An embedding $(\ell_1,\ell_2) \mapsto \ell_1 \kappa' + \ell_2 \eta'$ is {\bf primitive} if every lattice element that is a {\em real} linear combination of $\kappa'$ and $\eta'$ is also
an {\em integral} linear combination.}  $\Z^2 \hookrightarrow H^2(X';\Z)$.
By \cite{T}, such manifolds are essentially classified by their Duistermaat-Heckman functions, which can be any positive polynomial of degree at most two with even coefficients.  
In particular, there is a free circle action on a symplectic manifold $(M',\omega')$ with a proper moment
map $\Psi' \colon M' \to (-1,1)$  satisfying the criteria above
with Duistermaat-Heckman function $12 - 2t^2$.

Next, in Section~\ref{lfree}, we construct a locally free
circle action on a symplectic manifold $(\breve M, \breve \omega)$ with a proper moment map
$\breve \Psi \colon \breve M \to \R$ that satisfies the requirements of Theorem~\ref{theorem10} if we restrict
consideration to $(1,9)$. In particular, the Duistermaat-Heckman function of $\breve M$ is
$2 + \frac{1}{2}(t-5)^2$, and each reduced space is diffeomorphic to a generalized K3 surface with $5$ isolated
$\Z_2$ singularities.
Additionally, we endow $\breve M$ with two congenial complex structures.

The results in Section~\ref{lfree} rely heavily on our construction of 
symplectic orbifold forms on
generalized K3 surfaces with isolated $\Z_2$ singularities, which can be found in Section~\ref{K3}.
This in turn relies on Kobayashi and Todorov's classification of K\"ahler generalized K3 surfaces, and on the relationship between a $2$-dimensional complex orbifold with isolated $\Z_2$ singularities and its blow up, which is described in Section~\ref{orbifold}.

In Section~\ref{fixed}, we construct a  circle action on a symplectic manifold $(\tM,\tomega)$ with a proper moment
map $\tPsi \colon \tM \to (1-\tepsilon, 9 + \tepsilon)$, where $\tepsilon > 0$,
which satisfies the requirements of Theorem~\ref{theorem10} if we restrict consideration to $(1 - \tepsilon, 9 + \tepsilon)$;
in particular, the reduced space $\tM \modt$ is symplectomorphic to a tame K3 surface
$(\Hat X, \Hat I, (\Hat \sigma_\pm)_t)$ for all $t \in 5 \pm (4,4+ \tepsilon)$.
Additionally, $[(\Hat \sigma_\pm)_t] = \Hat \kappa_\pm - t \Hat \eta_\pm$,
where $\Hat \kappa_\pm, \Hat \eta_\pm$ induce a primitive embeddding $\Z^2 \hookrightarrow H^2(\Hat X;\Z)$.
This construction relies on techniques from \cite{TWa}: We use symplectic cutting twice on $\breve M$ to
construct a closed symplectic orbifold $M_\cut$ with 10 isolated $\mathbb{Z}_2$ singularities.
The manifold $\tM$ is an open subset of the blow up of $M_\cut$ at these $10$ points, with
a slightly perturbed symplectic form.  


The moment preimages $(\Psi')^{-1}(1 - \tepsilon, 1) \subset M'$ and $\tPsi^{-1}(1 - \tepsilon,1) \subset \tM$
are symplectic manifolds with free circle actions and proper moment maps to $(1  - \tepsilon,1)$; moreover, their reduced spaces are symplectomorphic to tame K3 surfaces, and both satisfy the technical assumption. 
Since their Duistermaat-Heckman functions agree, this implies that they are equivariantly symplectomorphic -- as long as we restrict
to the moment preimages of a small open set.  A nearly identical argument applies to the moment preimages
 $(\Psi')^{-1}(-1, -1+\tepsilon) \subset M'$ and $\tPsi^{-1}(9, 9 +   \tepsilon) \subset \tM$.
Therefore, as we show in Section~\ref{proof}, we can construct $M$ by gluing together $M'$ and $\tM$.
Finally, $M$ cannot be Hamiltonian, because every fixed point has both positive and negative weights.

\section{Free circle actions}
\labell{free}

In this section, we review the classification of  free Hamiltonian
circle actions on symplectic manifolds with reduced spaces symplectomorphic
to tame K3 surfaces (that also satisfy a technical condition) from \cite{T}.
Such manifolds are essentially classified by their Duistermaat-Heckman function, which can be any positive polynomial of degree at most two with even coefficients. 
The more precise statements below correspond to
\cite[Proposition 3.1]{T} and
\cite[Proposition 3.3]{T}, respectively.
We use the first proposition to construct the free portions of the manifolds described in our main theorems,
e.g., the part of the manifold constructed in Theorem 2 that lies over $(-1,1) \subset \R/10\Z$.
We use the second  to attach this portion to the rest of the manifold.

\begin{Proposition}\labell{freeunique2} 
Let the circle act freely on symplectic manifolds
$(M,\omega)$ and $(M',\omega')$ with proper moment maps
$\Psi \colon M \to (a,b)$ and $\Psi' \colon M' \to (a,b).$
Assume that, for all $t \in (a,b),$  the reduced spaces 
$M \modt  $ and 
$M' \modt $ are  
symplectomorphic 
to tame K3 surfaces $(X,I_t,\sigma_t)$
and $(X',I_t',\sigma_t')$, respectively; moreover,
\begin{itemize}
\item $\int_X \sigma_t^2 = \int_{X'}(\sigma'_t)^2$; and
\item $[\sigma_t] = \kappa - t \eta$ and $[\sigma'_t] = \kappa' - t \eta'$, where 
$\kappa, \eta $ and $\kappa',\eta'$ 
induce primitive embeddings
$\Z^2 \hookrightarrow H^2(X;\Z)$ and $\Z^2 \hookrightarrow H^2(X';\Z)$, respectively.
\end{itemize}
Then every $t \in (a,b)$ has a open neighborhood $U$ so that
$\Psi\inv(U)$ and $(\Psi')\inv(U)$ are equivariantly symplectomorphic.
\end{Proposition}

\begin{Proposition}\labell{existsfree2} 
Fix  a polynomial $P$ of degree at most two with even coefficients  that
is positive on $[a,b] \subset \R$.
Then there exists a free circle action on a symplectic
manifold $(M,\omega)$ with a proper moment map $\Psi \colon M \to (a , b)$
so that, for all $t \in (a, b)$, 
the reduced space
$M \modt $ is symplectomorphic to a tame
K3 surface $(X,I_t,\sigma_t)$; moreover,
\begin{itemize}
\item $\int_X \sigma_t^2 = P(t);$ and
\item $[\sigma_t] = \kappa - t \eta$, where $\kappa, \eta$  induce a primitive embedding $\Z^2 \hookrightarrow H^2(X;\Z)$.
\end{itemize}

\end{Proposition}

\section{Complex orbifolds with isolated $\Z_2$ singularities}
\label{orbifold}

In this section, we analyze the relationship between
a $2$-dimensional complex orbifold with isolated $\Z_2$ singularities and its
blow up, focusing on cohomology classes, holomorphic forms, and holomorphic  bundles.
Most of these facts are well known, and are true in a (much) more general context.
However, for completeness, we include here  elementary proofs for this simple case.

Let $X$ be a complex orbifold.
By definition, there exists an open cover $\{V_\alpha\}$ of $X$
and an {\bf orbifold chart}
$\{ \Tilde V_\alpha, \Gamma_\alpha, \phi_\alpha\}$ for each $\alpha$,
 where  $\Tilde V_\alpha$ is an
open subset of $\C^n$ that is invariant under the linear action of
a finite group $\Gamma_\alpha$, and $\phi_\alpha$ is a homeomorphism
from $\Tilde V_\alpha/\Gamma_\alpha$ to $V_\alpha$.
An \textbf{orbifold form} on $X$ 
is a  $\Gamma_\alpha$ invariant  form  on each orbifold chart $\Tilde V_\alpha$
so that
any two agree on overlaps.
We extend the notion of closed form, K\"ahler form, holomorphic form, etc.
to orbifolds analogously.

Let $\C^2/\Z_2$ be the quotient of $\C^2$ by the diagonal $\Z_2$ action.
Define an algebraic variety $Z := \{x \in \C^3 \mid x_1 x_3 = x_2^2\}$.
The  map from $\C^2/\Z_2$ to $Z$
that sends $[z_1,z_2]$ to $(z_1^2, z_1 z_2, z_2^2)$
is a continuous proper bijection to a locally compact Hausdorff space;
hence, it's a homeomorphism.
Moreover, by Luna \cite{L}, this homeomorphism  induces an isomorphism
between the sheaf of orbifold holomorphic functions on $\C^2/\Z_2$, and
the sheaf of holomorphic functions on $Z \subset \C^3$.
(In contrast, it does not induce an isomorphism between the sheaves
of smooth functions; the inverse homeomorphism from $Z$ to $\C/\Z_2$ is not smooth.)
Therefore, a $2$-dimensional complex orbifold
with isolated $\Z_2$ singularities is naturally
a complex analytic space, and maps between such orbifold are
maps of complex analytic spaces.

Conversely, any complex analytic space that is locally modelled
on  the variety $Z$ (or $\C^2$)  is naturally  a $2$-dimensional complex orbifold
with isolated $\Z_2$ singularities, and maps between such complex analytic
spaces are  maps of orbifolds.
This is an immediate consequence on the lemma below, adapted from \cite{KLM}.

\begin{Lemma}\label{l:lift}
Let $\Z_2$ act diagonally on $\C^2$ and let $\pi \colon \C^2 \to \C^2/\Z_2$
be the quotient map.
Let $U$ and $U'$ be  simply connected $\Z_2$ equivariant  subsets of $\C^2$,
and let $f \colon \pi(U) \to \pi(U')$ be a homeomorphism that induces an
isomorphism between the sheaves of holomorphic functions.
Then there  exists 
an equivariant  biholomorphic map
$\tilde f \colon U \to U'$ such that  $\pi \circ \tilde f = f \circ \pi$.
\end{Lemma}

\begin{proof}

If $0 \in U$ then
$f([0])= [0]$,
because the local homology group
$H_*(\C^2/\Z_2, \C^2/\Z_2  \smallsetminus \{[0]\};\Z)$ at $[0]$
is different from that of any other point.

The quotient map $\pi$ restricts to
a covering map from  $\C^2 \smallsetminus \{0\}$  onto $\C^2/\Z_2 \smallsetminus \{[0]\}$.
Hence, since
$U \smallsetminus \{0\}$ is simply connected,
there's a  lift of the restriction of $f$ to $\pi(U) \smallsetminus \{[0]\}$,
that is, a  holomorphic map $\tilde f_0 \colon U \smallsetminus \{0\}
\to U' \smallsetminus \{0\}$ such that
$\pi \circ \tilde f_0 = f \circ \pi$ on $U \smallsetminus \{0\}$.
Moreover, the composition of $\tilde f_0$ with the lift of the
restriction of $f^{-1}$ to $\pi(U') \smallsetminus \{[0]\}$
is either the identity or multiplication by $-1$.
Hence, $\tilde f_0$ is biholomorphic.
Similarly, $\tilde f_0$ is $\Z_2$ equivariant.

Next, fix a  compact neighborhood $B \subset U$ of $0$.
Since $\pi$ is proper, $\pi^{-1}(f(\pi(B)))$ is compact;
hence, $\tilde f_0(B \smallsetminus \{0\}) \subset \pi^{-1}(f(\pi(B)))$ is bounded.
Therefore, by  the Riemann extension theorem, $\tilde f_0$ has a unique holomorphic
extension $\tilde f \colon U \to U'$, which is the required 
lift of $f$. \end{proof}

Let $X$ be a $2$-dimensional complex orbifold with isolated $\Z_2$ singularities $p_1,\dots,p_k$.
We define the {\bf blow up} $q \colon \Hat{X} \to X$ of $X$ at $p_1,\dots,p_k$
to be its blow up as a complex analytic space. (See, for example, \cite{Fi}.)
Blow ups are  unique up to canonical isomorphism. Moreover, the map  $q$  
restricts to a biholomorphism from $q^{-1}(X \smallsetminus \{p_1,\dots,p_k\})$
to $X \smallsetminus \{p_1,\dots,p_k\}$.
Hence, since blowing up commutes with restriction to open subspaces, the blow up of $X$ is determined by the  lemma below.
In particular, $q$ is proper and the {\bf exceptional divisor} $E_i := q\inv(p_i)$ is biholomorphic to $\CP^{1}$ for all $i$.

\begin{Lemma}
Let $\C^\times$ act on $(\C^2 \ssetminus \{0\}) \times \C$ by
$$\lambda \cdot (z_1,z_2;u) = (\lambda z_1,  \lambda z_2; \lambda^{-2} u).$$
The blow up of $\C^2/\Z_2$ is 
the complex surface
$$\Hat{ \C^2/\Z_2} := (\C^2 \ssetminus \{0\}) \times_{\C^\times} \C,$$
and the map
that
takes $ [z_1,z_2;u] \in \Hat{\C^2/\Z_2}$  to $[ \sqrt{u} z_1,  \sqrt{u} z_2] \in \C^2/\Z_2$.
\end{Lemma}

Note that this map is not smooth.

\begin{proof}
The  blow up of $Z  := \{x \in \C^3 \mid x_1 x_3 = x_2^2\}$ at $0$ is
the algebraic variety
$$\Hat{Z} := \{ (x, [y]) \in \C^3 \times \CP^2 \mid
y_1 y_3 = y_2^2 \mbox{ and }  x_i y_j = x_j y_i \ \forall \ i,j  \},$$
and  the map  $q \colon \Hat{Z} \to Z$ defined by $q(x,[y]) =x$.
Define  a holomorphic map from  $\Hat{\C^2/\Z_2}$ to $\Hat Z$ by 
$$[z_1,z_2;u] \mapsto ((u z_1^2, u z_1 z_2, u z_2^2), [ z_1^2, z_1 z_2, z_2^2]).$$
An inverse map is given by $((x_1,x_2,x_3),[y_1,y_2,y_3]) \mapsto
(\sqrt{y_1}, \frac{y_2}{\sqrt{y_1}}; \frac{x_1}{y_1}) $ for $y_1 \neq 0$,
and by 
 $((x_1,x_2,x_3),[y_1,y_2,y_3]) \mapsto
(\frac{y_2}{\sqrt{y_3}}, \sqrt{y_3}; \frac{x_3}{y_3}) $ for $y_3 \neq 0$.
Under the resulting identifications of $\Hat{\C^2/\Z_2}$ with
$\Hat{Z}$ and $\C^2/\Z_2$ with $Z$, the  map $q \colon \Hat{Z} \to Z$
takes $ [z_1,z_2;u] \in \Hat{\C^2/\Z_2}$  to $[ \sqrt{u} z_1,  \sqrt{u} z_2] \in \C^2/\Z_2$. \end{proof}

We are now ready to demonstrate some relationships between the cohomologies of $X$ and $\Hat X$.

\begin{Lemma} \label{lcoho} 
Let $q:\Hat X \rightarrow X$ be the blow up  of a $2$-dimensional
complex orbifold
 at isolated $\mathbb{Z}_2$ singularities $p_1,\dots,p_k$.
Let $I(X)$ be  the kernel of the restriction map
$H^2(\Hat X;\Z) \to H^2( \cup_i E_i; \Z)$,
where $E_i := q^{-1}(p_i)$ for each $i$.
\begin{enumerate}
\item The map $q$ induces an isomorphism $q^* \colon H^2(X;\mathbb{Z}) \to I(X)$. 
\item 
The natural inclusion  $\hat \iota \colon \Hat X \ssetminus \cup_i E_i \to \Hat X$ 
induces an injection  $\hat \iota^* \colon I(X) \to H^2(\Hat X \ssetminus \cup_i E_i;\Z)$.
\end{enumerate}
\end{Lemma}

\begin{proof}
Let $U_i$ be an open  neighborhood of $p_i$ that lifts to a  disk in
an orbifold chart  for each $i$.
Then  $\{p_i\}$ is a deformation retract of $U_i$, and $E_i$ is
a deformation retract of $\Hat U_i := q^{-1}(U_i) $, for all $i$.
Moreover,  $q$ restricts to a homeomorphism from $\Hat X \smallsetminus \cup_i E_i$
to $X \smallsetminus \{p_i\}_{i}$. Hence,  the pullback map $q^* \colon
H^*(X, \{p_i\}_{i};\Z) \to 
H^*(\Hat X, \cup_i E_i;\Z)$ is an isomorphism.
Since $H^1(\cup_i E_i;\Z) = H^1(\{p_i\}_i;\Z) = H^2( \{p_i\}_i;\Z) = 0$,
the natural maps $H^2(\Hat X , \cup_i E_i;\Z) \to I(X)$ 
and
$H^2(X, \{p_i\}_i;\Z) \to H^2(X;\Z)$
are isomorphisms;
the first claim follows immediately.
Since $H^1( \Hat U_i \ssetminus  E_i;\Z) = H^1( \R\P^{3};\Z) = \{0\},$
the second claim follows from the 
long exact Mayer-Vietoris sequence
for the pair $(\Hat X \smallsetminus  \cup_i E_i,\, \cup_i \Hat U_i)$. 
\end{proof}

Given an $n$-dimensional  manifold  $N$ and a nonnegative integer $m$, let $\Omega^m_c(N)$ be the vector space of
smooth forms of degree $m$ on $N$ with compact support.
An {\bf $\mathbf m$-current} is a continuous (in the sense of distributions)
linear functional  $T \colon \Omega^{n-m}_c(N) \to \R$.
Let $\mathcal C_m(N)$ denote
the vector-space of $m$-currents on $N$.
Define a differential  $\del \colon \mathcal C_m(N) \to \mathcal C_{m+1}(N)$ by
$(\del T)(\psi) := (-1)^{m+1} T(d \psi)$ for all $\psi \in \Omega_c^{n-m-1}(N)$;
clearly, $\del^2 = 0$.

A (smooth) $m$-form $\eta \in \Omega^m(N)$ defines an  $m$-current by $\psi \mapsto \int_N \eta \wedge \psi$ for all $\psi \in \Omega_c^{n-m}(N)$;
we identify this current with $\eta$.
By Stokes' Theorem, this identification intertwines the differentials on $\Omega^*(N)$
 and $\mathcal C_*(N)$.
Moreover, 
every closed current is homologous to a differential form,  and a form
is the boundary of a current exactly if it is the boundary of another  form
 \cite[Theorem 14]{deRham}. 
Therefore, every closed current represents a unique cohomology class.

Now let $q \colon \Hat X \to X$ be the blow up  of a $2$-dimensional
complex orbifold $X$ at isolated $\Z_2$ singularities $p_1,\dots,p_k$.
Let $\hat \iota \colon \Hat X \ssetminus \cup_i E_i \to \Hat X$
be the inclusion map, where $E_i := q^{-1}(p_i)$ for all $i$.
Although $q$ is generally not smooth, the composition $q \circ \hat \iota$
is smooth.  
Hence, given  an orbifold form $\eta \in \Omega^k(X)$,
the pullback $(q \circ \hat \iota)^* \eta$
is a smooth form on $\Hat X \ssetminus \cup_i E_i$.
We say that $\eta$ {\bf defines a current} on $\Hat X$ if  the map
$\Omega_c^{n-k}(\Hat X) \to \mathbb{R}$ given by  $\psi \mapsto \int_{\Hat X \ssetminus \cup_i E_i } \hat \iota^*\psi \wedge (q \circ \hat \iota)^*\eta$ 
is a (well-defined) continuous linear function.
Our next lemma gives the cohomology class of this current when $\eta$ is closed.

\begin{Lemma} \label{current}
Let $q \colon \Hat X \to X$ be the blow up of a
$2$-dimensional complex orbifold  at isolated $\Z_2$ singularities.
Let $I(X)_\R$ be the kernel of the restriction map
$H^2(\Hat X; \R) \to H^2(\cup_i E_i;\R)$, where
$E_1,\dots,E_k$ are the exceptional divisors.
If a closed orbifold form $\eta \in \Omega^2(X)$ defines a closed current on $\Hat X$
in a class $\epsilon \in I(X)_\R$, then $q^*([\eta])=\epsilon$. \end{Lemma}

\begin{proof}
Let $\hat \iota \colon \Hat X \ssetminus \cup_i E_i \to \Hat X$
be the inclusion map, and let $\eta$ define
a closed current $T$  on $\Hat X$.
The restriction of $T$
to $\Hat X \ssetminus \cup_i E_i$ is the current defined by
$(q \circ \hat \iota)^* \eta \in \Omega^2(\Hat X \ssetminus \cup_i E_i).$
Hence, $\hat \iota^* [T] = [ (q \circ \hat \iota)^* \eta] =
\hat \iota^* q^* [\eta]$, and so  the claim follows from Lemma~\ref{lcoho}(2).
\end{proof}

Next, we demonstrate a  relationship between holomorphic forms on $X$ and $\Hat X$.

\begin{Lemma} \label{l38}
Let $q \colon \Hat X \to X$ be the blow up of a $2$-dimensional
complex orbifold at isolated $\Z_2$ singularities.
Given a nonvanishing holomorphic form $\Hat \mu \in \Omega^{2,0}(\Hat X)$, there exists
a nonvanishing holomorphic orbifold form $\mu \in \Omega^{2,0}(X)$ such that
 $q^*\big([\mu]\big) = \big[\Hat\mu\big]$.
\end{Lemma}

\begin{proof}
Let $p_1,\dots,p_k \in X$ be the isolated $Z_2$ singularities.
Let $\hat \iota \colon \Hat X \ssetminus \cup_i E_i \to \Hat X$ denote
the natural inclusion, where $E_i := q\inv(p_i)$ for each $i$.
Since $q$ restricts to a biholomorphism from $\Hat X \ssetminus \cup_i E_i$
to $X \ssetminus \{p_i\}_i$, 
there exists a holomorphic $2$-form on $X \ssetminus \{p_i\}_i$  whose pullback under this biholomorphism  is $\hat \iota^* \Hat \mu$. 
 By Lemma~\ref{extend} below
we can extend this form to a holomorphic $2$-form $\mu$ on $X$
such that $(q \circ \hat \iota)^* \mu = \hat \iota^* \Hat \mu$.
Moreover, holomorphic $2$-forms on complex surfaces are closed; hence, $\hat\iota^* q^* [\mu] = \hat \iota^* [\Hat \mu]$.
Since each $E_i$ is a complex curve,  
the restriction of $[\Hat \mu]$ to $\cup_i E_i$ vanishes.
Thus, 
 $q^*\big([\mu]\big) = \big[\Hat\mu\big]$ by Lemma~\ref{lcoho}(2).

Let $U_i$ be  an open neighborhood of $p_i$ that lifts to an open set $\Tilde U_i \subset \C^2$ in an orbifold chart around $p_i$.
Identify $U_i$ with $\Tilde U_i/\Z_2$ and
$\Hat U_i := q^{-1}(U_i)$ with the preimage of $\Tilde U_i/\Z_2$  in $\Hat{\C^2/\Z_2}$.
Let $\Tilde \mu := \Tilde f dz_1 \wedge dz_2 \in \Omega^2(\Tilde U)$
be the pullback of $\mu$ to this orbifold chart, where $\Tilde f$ is a  $\Z_2$ invariant holomorphic function. 
Since $(q \circ \hat \iota)^*\mu = \hat \iota^* \Hat \mu$,
$$ \textstyle \Hat \mu  = \frac{1}{2}\Tilde f( \sqrt{u} z_1, \sqrt{u} z_2) \left(  z_1 du \wedge d z_2 +  z_2 d z_1 \wedge du + 2u dz_1 \wedge d z_2 \right)$$
on $\Hat U_i \smallsetminus E_i$.
Since $\Tilde \mu$ and $\Hat \mu$ are both continuous, this equation
holds on all of $\Hat U_i$.
Therefore, since  $\Hat \mu$ doesn't vanish on $\Hat U_i$,  
the form $\mu$ doesn't vanish on $\U_i$.
Therefore, $\mu$ is nonvanishing.

\end{proof}

Here, we used the following variant of Hartogs' extension theorem.

\begin{Lemma}\labell{extend}
Let $X$ be a complex orbifold of dimension at least $2$ with isolated singularities.
Then every holomorphic $m$-form $\mu'$ on the regular part
of $X$ can be extended uniquely to a holomorphic
orbifold $m$-form $\mu$ on $X$.
\end{Lemma}

\begin{proof}
Let $(\Tilde V, \Gamma)$ be a local orbifold chart near  an
isolated singularity $p \in X$.
Let $\Tilde \mu' := \sum_{i_1 < \dots <  i_m} f'_{i_1 \dots i_m}
 dz_{i_1} \wedge \dots \wedge d z_{i_m} \in \Omega^{m,0}(\Tilde V \ssetminus \{0\})$
be the pullback of $\mu'$ to this chart,
where each $f'_I := f'_{i_1\dots i_m}$ is a holomorphic function on $\Tilde V \ssetminus \{0\}$.
By Hartogs' extension theorem, we can extend each $f'_I$ uniquely to a holomorphic function $f_I$ on $\Tilde V$.
Thus, $\Tilde \mu := \sum_{i_1 < \dots <  i_m} f_{i_1 \dots i_m}
 dz_{i_1} \wedge \dots \wedge d z_{i_m}$
is the unique holomorphic extension of $\Tilde \mu'$ to $\Tilde V$.
Taken together with the fact that $\Tilde \mu'$ is $\Gamma$ invariant,
this implies that
$\Tilde \mu$ is also invariant.
\end{proof}

Finally, we demonstrate a  relationship between holomorphic  bundles on $X$ and $\Hat X$.

\begin{Definition}\label{d37}
Let $\cL$ be a holomorphic $\C^\times$  bundle over a complex orbifold $X$ and let $h$ be a hermitian metric on the associated line bundle containing $\cL$.
There exists a real orbifold form in $\Omega^{1,1}(X)$ with local expression 
$$\zeta = \textstyle \frac{\sqrt{-1}}{2 \pi} \del \delbar s^*(\Psi)$$
for every local holomorphic section $s$ of $\cL$, where
$\Psi(m) := \ln |m|^2$ for all $m \in \cL$. (Here, we identify $\cL$ with a subset of the associated line bundle.)
The \textbf{Chern curvature} of $h$ is the orbifold form $2 \pi \sqrt{-1} \, \zeta \in \Omega^{1,1}(X)$. 
The \textbf{Euler class} of $\cL$ is the cohomology class $e(\cL) := -[\zeta] 
\in H^{1,1}(X)$, which does not depend on the choice of the hermitian metric $h$.
\end{Definition}

\begin{Lemma} \label{l39}
Let $q \colon \Hat X \to X$ be the blow up of a $2$-dimensional 
complex orbifold at isolated $\Z_2$ singularities $p_1,\dots,p_k$.
Given a holomorphic $\C^\times$ bundle $\Hat \cL$ over $\Hat X$
there exists a holomorphic $\C^\times$ bundle $\cL$ over $X$
so that $e(\cL) \in H^{1,1}(X)$ is the unique class with
$$q^*(e(\cL)) = e(\Hat \cL)  + 
\textstyle
\sum_i 
\frac{1}{2}
\cE_i 
\int_{E_i}  \!  e(\Hat \cL),$$
where 
$\cE_i  \in H^{1,1}(\Hat X;\R)$ is the Poincar\'e dual to 
$E_i := q\inv(p_i)$ for each $i$.
Moreover, $\cL$ is a manifold exactly if $\int_{E_i} \! e(\Hat \cL)$ is odd for all $i$.
\end{Lemma}

\begin{proof}
Fix $n \in \Z$.
Let  
$\Hat{\C^2/\Z_2} \to \C^2/\Z_2$ 
be the blow up of $\C^2/\Z_2$ at $[0]$, and let
$E := (\C^2 \smallsetminus \{0\}) \times_{\C^\times} \{0\}$ be the exceptional divisor.
Define a holomorphic
$\C^\times$  bundle $\Hat D_n$ over 
$ \Hat{\C^2/\Z_2}$ 
with 
$\int_E e(\Hat D_n) = n$
by
$$ \Hat D_n  := \left( \C^2 \ssetminus \{0\} \right) \times_{\C^\times} (\C \times \C^\times),$$ 
where $\C^\times$ acts by $$\lambda \cdot (z_1,z_2;u; w) = (\lambda z_1, \lambda z_2; \lambda^{-2} u; \lambda^n w).$$
Similarly,  define a holomorphic $\C^\times$ bundle $D_n$ over 
$\C^2/\Z_2$  by 
$$ D_n := \C^2 \times_{\Z_2} \C^\times,$$
where $\Z_2$ acts by  $$\pm 1 \cdot (z_1,z_2; w) = ( \pm z_1, \pm z_2; (\pm 1)^n w).$$
Clearly, $D_n$ is a manifold exactly if $n$ is odd.
The natural biholomorphism from
$\Hat {\C^n/\Z_2} \ssetminus E$ to $\C^2/\Z_2 \ssetminus \{[0]\}$
lifts to an isomorphism of bundles
\begin{equation}\labell{iso1}
\Hat D_n \big|_{ \big(\Hat{C^2/\Z_2} \ssetminus E\big)} \simeq D_n \big|_{\big(\C^2/\Z_2 \ssetminus \{[0]\}\big)};
\end{equation}
it sends
$[z_1,z_2;u;w]$ to $[\sqrt{u} z_1,\sqrt{u} z_2; (\sqrt{u})^n w]$.

Fix $i \in \{1,\dots,k\}$,
and
let $U_i$ be  an open neighborhood of $p_i$ that lifts to a disk in an orbifold chart.
Then $E_i$ is a deformation retract of  $\Hat U_i := q^{-1}(U_i)$,
and so
$e\big( \Hat D_{\int_{E_i} \!  e(\Hat \cL)}\big)\big|_{\Hat U_i}  
= e(\Hat \cL)|_{\Hat U_i}.$  
Moreover, 
by the lemma below, $H^1(\Hat U_i,\mathcal O) = 0$,
where $\mathcal O$ denote the sheaf of holomorphic functions.
Therefore, there is an isomorphism of bundles over $\Hat U_i$,
\begin{equation}\labell{iso2}
\Hat D_{\int_{E_i} \!  e(\Hat \cL)}\big|_{\Hat U_i}
\simeq 
 \Hat \cL\big|_{\Hat U_i}. 
\end{equation}

Since $q$ restricts to a biholomorphism from
$\Hat X \ssetminus \cup_i E_i$ to $X \ssetminus \{p_i\}_i$,
there exists a holomorphic
$\C^\times$ bundle $\cL'$ over $X \ssetminus \{p_i\}_i$ such that
$q^*\cL' \simeq \hat \iota^* \Hat \cL$,
where $\hat \iota \colon \Hat X \ssetminus \cup_i E_i \to \Hat X$
is the inclusion map.
Moreover,  by \eqref{iso1} and \eqref{iso2},
the bundles $\cL'$ and  $D_{\! \int_{E_i} \!  e(\Hat \cL)}$  
are isomorphic 
over $U_i \ssetminus \{p_i\}$ for each  $i$.
Therefore, by gluing together the bundle $\cL'$
and the 
 bundle  $D_{\! \int_{E_i} \! e(\Hat \cL)}$  over $U_i$
for each $i$, 
we  construct a holomorphic $C^\times$
bundle $\cL$ over $X$ so that
$(q \circ \hat \iota)^* \cL \simeq \hat \iota^* \Hat \cL$.
Since each $\cE_j$ vanishes on $\Hat X\ssetminus \cup_i E_i$,
this implies that 
$$\hat \iota^* q^* e(\cL) = \hat \iota^* e(\Hat \cL) = \hat \iota^*\big( e(\Hat \cL)
+ \textstyle \sum_i \frac{1}{2} \cE_i \int_{E_i} \! e(\Hat \cL) \big).$$
Moreover, since $\int_{E_j} \! \cE_i = -2 \delta_{ij}$ for all $i, j$,
the class
$e( \Hat \cL) + \sum_i \frac{1}{2} \cE_i \int_{E_i} \!  e( \Hat \cL) $
vanishes when restricted to $\cup_i E_i$.
Therefore, the claim follows from Lemma~\ref{lcoho}.
\end{proof}

Here, we used the following calculation.

\begin{Lemma}
Let $q \colon \Hat{ \C^2/\Z_2} \to \C^2/\Z_2$ be the  blow up of $\C^2/\Z_2$ at $[0]$.
Given $\epsilon > 0$, let $U := \{ [z] \in \C^2/\Z_2 \mid |z|^2 < \epsilon\}$
and $\Hat U := q^{-1}(U)$.  
Then $H^{1}(\Hat U,\cO) = 0$,
where $\cO$ denotes the
sheaf of holomorphic functions.
\end{Lemma}

\begin{proof}
In coordinates,  $\Hat U  = \{ [z_1,z_2;u ] \mid |u||z|^2 < \epsilon \}$.
Define
$$\Hat W_1 :=  \{ [z_1,z_2;u] \in \Hat U \mid z_1 \neq 0 \} \quad \mbox{and} \quad
\Hat W_2 :=  \{ [z_1,z_2;u] \in \Hat U \mid z_2 \neq 0 \}.$$
Define maps 
$\phi_1$ and $\phi_2$ from  $\C^2$ to $\Hat{ \C^2/\Z_2}$ by 
$\phi_1(z_1,z_2) := [1,z_1;z_2]$
and $\phi_2(z_1,z_2) := [z_1, 1; z_2]$.
Then each
 $\phi_i$ restricts to a biholomorphism
from $A := \{ x \in \C^2 \mid |z_2|(1 + |z_1|^2) < \epsilon \}$
to $\Hat W_i$, and from $A' := A \cap (\C^\times \times\C)$ to $\Hat W_1 \cap \Hat W_2$.

First we claim $H^i(\Hat U,\cO)=H^i(\{\Hat W_1, \Hat W_2\},\cO)$, where $H^i(\{\Hat W_1, \Hat W_2\},\cO)$ denotes the $\textrm{\v{C}ech}$ cohomology of the cover $\{\Hat W_1,\Hat W_2\}$ with coefficients in $\cO$.
Since the function $t \mapsto \ln \epsilon - \ln(1+e^{2t})$ has negative  second derivative,
the set $A$ is a logarithmically convex (relatively) complete Reinhardt domain \cite[{\bf I.3}]{A}.
Hence, by \cite[{\bf I.5.15} and {\bf I.5.16}]{A}, $A$ is holomorphically convex.
Since every open subset of $\mathbb{C}^2$ satisfies the remaining conditions of Definition of 5.1.3 of \cite{H}, this implies that 
$A$ is a Stein manifold; hence, $\Hat W_1$ and $\Hat W_2$ are as well.
The claim now follows by Theorem 7.4.1 of \cite{H}.

Hence, to complete the proof, it is enough to prove $H^1(\{\Hat W_1,\Hat W_2\}, \cO)$
$= 0$.
So consider $\tilde{f} \in C^1(\{\Hat W_1,\Hat W_2\},\cO)= \cO(\Hat W_1 \cap \Hat W_2)$.
Since $A'$ is a Reinhardt domain containing $\C^\times \times \{0\}$,
 there exist unique constants
$f_{ij} \in \C$ such that $A'$ is contained in the domain of convergence of the Laurent series
$$f:=\sum_{i=-\infty}^\infty \sum_{j=0}^\infty f_{ij} z_1^i z_2^j$$
and the series converges to the holomorphic function $\phi_1^*(\tilde f)$ on $A'$ \cite[\bf I.6.3]{A}.
Therefore, $A'$ is contained in the domains of convergence of the Laurent series
$$ g:=\sum_{i=0}^\infty \sum_{j=0}^\infty f_{ij} z_1^i z_2^j \quad \textrm{and} \quad g-f=-\sum_{i=-\infty}^{-1} \sum_{j=0}^\infty f_{ij} z_1^i z_2^j.$$
Since $\phi_1^{-1} \circ \phi_2$ is an isomorphism from $A'$ to itself, this implies that $A'$ is contained in the domain of convergence of the Taylor series
$$h:=(\phi_1^{-1} \circ \phi_2)^*(g-f)=-\sum_{i=-\infty}^{-1} \sum_{j=0}^\infty f_{ij} z_1^{2j -i} z_2^j=-\sum_{j=0}^\infty \sum_{k=2j+1}^\infty f_{kj} z_1^k z_2^j.$$
Since $A$ is open and $A'=A \cap (\mathbb{C}^\times \times \mathbb{C})$, this implies that $A$ is contained in the domains of convergence of the Taylor series $g$ and $h$.
Therefore, there exist $\tilde g \in \cO(\Hat W_1)$ and $\tilde h \in \cO(\Hat W_2)$ such that $g$ converges to $\phi_1^*(\tilde g)$ and $h$ converges to $\phi_2^*(\tilde h)$. 
(See, e.g., \cite[{\bf I.2A} and {\bf I.2B}]{A}.)
Hence, $\tilde f = \tilde g - \tilde h$  on $\Hat W_1 \cap \Hat W_2$, that is,
$\tilde f = \delta(\tilde g,\tilde h)$, as required. \end{proof}

\section{Generalized K3 surfaces}
\labell{K3}

In this section, we use Kobayashi and Todorov's classification of K\"ahler generalized K3 surfaces
to obtain symplectic orbifold forms on
generalized K3 surfaces with isolated $\Z_2$ singularities.
These spaces will serve as the reduced spaces of the six-dimensional symplectic manifolds with Hamiltonian circle actions that
we construct in the next two sections. More concretely, we prove Proposition~\ref{propKT} below.

Recall that a \textbf{K3 surface} is a  closed, connected complex surface $(\Hat X, \Hat I)$
with $H^1(\Hat X;\mathbb{R})=\{0\}$ and trivial canonical bundle.
A closed complex surface $(X,I)$ with at worst simple singular points and with minimal resolution $q \colon \Hat X \to X$ is a {\bf generalized K3 surface} if $\Hat X$ is a K3 surface.
Here, a proper morphism $q \colon Y \to X$ is a \textbf{resolution} (of the singularities) of $X$ if $Y$ is non-singular and
$q$ restricts to an isomorphism from $Y \setminus q^{-1}(X_\mathrm{sing})$ to $X \setminus X_{\textrm{sing}}$.
A resolution $q \colon Y \to X$ of  $X$ is a \textbf{minimal resolution} if for every resolution $q' \colon Y' \to X$ of $X$ there exists a unique morphism $\Xi:Y' \to Y$ such that $q' = q \circ \Xi$.
In particular, if $X$ is a generalized K3 surface with isolated $\mathbb{Z}_2$ singularities, then the blow up $q \colon \Hat X \to X$ of $X$ at the singular points is the minimal resolution of $X$, and so $\Hat X$ is a K3 surface.
Here, to see that $q$ is the minimal resolution, note that $q^{-1}(p_i)$ does not contain a $(-1)$-curve for any $i$; see, e.g., \cite{I}.

\begin{Proposition}\labell{propKT}
Fix $k \in \{0,1,\dots,9\}$, a positive even integer $A$,
and an even integer $B > -\frac{1}{2} k$.
There exist a generalized K3 surface $(X,I)$ with $k$ isolated $\Z_2$ singularities,	
a K\"ahler orbifold form $\zeta \in \Omega^{1,1}(X)$,
and a real symplectic orbifold form $\nu \in \Omega^{2,0}(X) \oplus \Omega^{0,2}(X)$
so that $(\nu + t \zeta ,\nu + t \zeta) = A + ( B + \frac{1}{2} k) t^2$ for all $t \in \R$. 
Moreover, $q^*([\nu + 2\ell \zeta])$ is primitive\footnote{In particular,
it lies in $H^2(\Hat X;\Z) \subset H^2(\Hat X;\R)$.} for all $\ell \in \Z$, where $q \colon \Hat X \to X$ is the blow up of $X$ at the singular points.
Finally, there exists a holomorphic $\C^\times$ bundle $\cL$ over $X$ 
with Euler class $- [\zeta]$
so that $\cL$ is a manifold.
\end{Proposition}

Before proving this proposition, we introduce some terminology.
Let $q \colon \Hat X \to X$ be the minimal resolution of a generalized K3 surface $X$.
Let $\delta_1,\cdots,\delta_k \in H^{1,1}(\Hat X) \cap H^2(\Hat X;\mathbb{Z})$ be the Poincar\'e duals of the $(-2)$-curves contracted 
by  $q$.
By definition, the \textbf{root system} of $X$ is
$$R(X):=\Big\{\delta=\sum_{i=1}^k a_i \delta_i 
\, \Big| \,  
a_i \in \mathbb{Z} \ \mbox{and} \  (\delta,\delta)=-2\Big\}
\subset H^{1,1}(\Hat X) \cap H^2(\Hat X;\mathbb{Z}).
$$
Let $I(X) \subset H^2(\Hat X; \Z)$ be the set of all classes orthogonal to $R(X)$. 
Clearly, the image of  the pullback map $q^* \colon
H^2(X;\Z) \to H^2(\Hat X ;\Z)$ is contained in $I(X)$.
Similarly, $H^{2,0}(\Hat X)$ is contained in $I(X) \otimes \C$.

The {\bf K3 lattice} $L$ is the even unimodular lattice with signature $(3; 19)$; let $L_{\R} := L \otimes_{\Z} {\R}.$ 
Consider the cup product pairing $(\cdot, \cdot)$ on $\Hat X$.
Since  $(\Hat X,\Hat I)$ is a K3 surface,
there exists an isometry from $H^2(\Hat X;\mathbb{Z})$ to $L$ \cite[Proposition VIII.3.2]{BPV}, that is, 
an isomorphism of groups that intertwines the symmetric bilinear forms.
A {\bf marking} of $X$ is
a metric injection $\phi \colon I(X)\rightarrow L$ that can be extended to an isometry $\Hat\phi:H^2(\Hat X;\mathbb{Z}) \rightarrow L$. 
A \textbf{marked generalized K3 surface} is a triple $(X,I,\phi)$ 
where $(X,I)$ is a generalized K3 surface and $\phi$ is a marking. 

Let $$\Omega:=\{\alpha \in L_{\mathbb{C}} \mid (\alpha,\alpha)=0 \mbox{ and } (\alpha,\overline{\alpha})>0\}/\C^{\times}$$ be the classical period domain.
Since $H^{2,0}(\Hat{X}) \simeq \C$, there is a well-defined {\bf period point}
associated to any marked generalized K3 surface $(X,I, \phi)$:
$$\tau_1(X,I, \phi) := [\phi(\alpha_{\Hat X})] \in \Omega,$$
where $\alpha_X \in H^{2,0}(\Hat X)$ is any non-zero class.

A \textbf{polarization} on $X$ is an element of a subset $V_P^+(X)$ of $I(X)$ defined by
$$
V_P^+(X):=\left\{\kappa \in \overline{V_P^+(\Hat X)} \ \Bigg| \
\begin{array}{l}
\mbox{Given } \delta\in H^{1,1}(\Hat X)\cap H^2(\Hat X;\mathbb{Z}) \\
\mbox{with } (\delta,\delta)=-2,\\
(\kappa,\delta)=0 \mbox{ if and only if } \delta \in R(X)
\end{array}\right\},
$$
where $\overline{V_P^+(\Hat X)}$ is the closure of
$$
V_P^+(\Hat X):=\left\{\kappa \in H^{1,1}(\Hat X) \cap H^2(\Hat X;\mathbb{R}) \ \Bigg| \
\begin{array}{l} (\kappa,\kappa) >0  \textrm{ and }(\kappa,\delta) > 0 \\ \mbox{for all }  \mbox{effective} \\ (-2)\textrm{-classes }\delta \textrm{ on } \Hat X 
\end{array} \right\}.
$$
A \textbf{marked polarized generalized K3 surface}  $(X,I,\epsilon,\phi)$
is a marked generalized K3 surface $(X,I,\phi)$ and a polarization $\epsilon$ on $X$.

Let
$$K\Omega:=\{(\kappa,[\alpha])\in L_{\mathbb{R}} \times \Omega\mid (\kappa, \kappa) > 0 \mbox{ and } (\kappa, \alpha) = 0 \}$$ 
be the {\bf polarized period domain}.
There is a {\bf polarized period point} associated to each marked polarized generalized K3 surface $(X,I, \epsilon,\phi)$:
$$\tau_2(X,I, \epsilon,\phi) := (\phi(\epsilon), [\phi(\alpha_{\Hat X})]) \in K\Omega,$$
where again $\alpha_{\Hat X} \in H^{2,0}(\Hat X)$ is any non-zero class.
As Morrison \cite{Mo} proved, $\tau_2$ induces a bijection between isomorphism classes of marked polarized generalized K3 surfaces and $K\Omega$. 
Here two marked polarized generalized K3 surfaces 
$(X_1,I_1,\epsilon_1,\phi_1)$ and $(X_2,I_2,\epsilon_2,\phi_2)$ are isomorphic
if $\phi_1(\epsilon_1)=\phi_2(\epsilon_2)$ and there is a biholomorphism $f \colon X_1 \to X_2$ such that
$\phi_2 \circ q_2^*=\phi_1 \circ q_1^* \circ f^*$.
We need the following theorem, due to Kobayahsi and Todorov, which is a refinement of
the fact that $\tau_2$ is surjective.

\begin{refTheorem}\labell{KT}
Given $(\kappa, [\alpha]) \in K\Omega$, there exists a marked
polarized generalized K3 surface $(X,I, \epsilon,\phi)$ such that
$\tau_2(X,I, \epsilon,\phi) = (\kappa,[\alpha])$.
Moreover,  there exists a  K\"ahler orbifold form $\zeta \in \Omega^{1,1}(X)$
that  defines a closed current on $\Hat X$ in the class 
$\epsilon$, where $q \colon \Hat X \to X$ is the minimal resolution. 
\end{refTheorem}

\begin{proof}
The first claim is a theorem of Morrison \cite{Mo}; see also  \cite[Theorem C, p.\ 347]{KT}.
By Theorem 3 of \cite[p.\ 353]{KT}, $\epsilon$ contains a K\"ahler form on $X$ in the sense of Fujiki-Moishezon.
Thus, Theorem 1 of \cite[p.\ 348]{KT} implies
that there exists a Ricci-flat K\"ahler-Einstein orbifold form on $X$ which defines a closed current 
on $\Hat X$ in the cohomology class  $\epsilon$.
(Alternatively,
see \cite[Theorem 13, p.\ 284]{Kob} and the definition of ``K\"ahler polarization" \cite[p.\ 279]{Kob}.)
\end{proof}

\begin{Remark}
Some authors, e.g. \cite{Loo}, only consider one component of the polarized period domain and restrict marked polarized K3 surfaces
accordingly; Kobayashi and Todorov follow this tradition.
In contrast, following \cite{BPV}, we consider both components. 
This change is inconsequential,
because the isomorphism of $L$ that sends
$x$  to $-x$ induces a diffeomorphism of $K\Omega$ that exchanges the two components.
Similarly, 
instead of restricting to $\kappa \in L_\R$ such that $(\kappa,\kappa) = 1$,
we allow all $\kappa$ such that $(\kappa,\kappa) > 0$.
\end{Remark}

We will use the following corollary of Theorem~\ref{KT}.

\begin{Corollary}\labell{corKT}
Fix a negative-definite primitive sublattice $K \subset L$ and
$\kappa, \beta \in K^\perp \subset L_\R$ such that 
$(\kappa, \kappa) >0$, 
$(\beta, \beta)  > 0$,
and $(\kappa,\beta) =0$.
There exists a marked polarized  generalized K3 surface $(X,I,\epsilon, \phi)$  
such that $\kappa = \phi(\epsilon)$,  $\beta \in \phi( H^{2,0}(\Hat X) \oplus H^{0,2}(\Hat X))$,
and $\Hat\phi(R(X)) = \{d \in K \mid (d,d) = -2 \}$.
Moreover, there exists a  K\"ahler orbifold form $\zeta \in \Omega^{1,1}(X)$ that
defines a closed current on $\Hat X$ in the class $\epsilon$.
Here, $q \colon \Hat X \to X$ is the minimal resolution of $X$ and $\Hat \phi \colon H^2(\Hat X; \Z) \to L$
is any isometry extending $\phi$.
\end{Corollary}

\begin{proof}
Let $W \subset L_\R$ be the subspace generated by  $\kappa$,  $\beta$, and $K$.
For each $\ell \in L \ssetminus W$, the projective space  $\P(W^\perp \cap \ell^\perp)$ is a codimension one submanifold of $\P (W^\perp)$.
Since $L \ssetminus W$ is countable, the complement of the union $\cup_{\ell \in L \ssetminus W} \P (W^\perp \cap \ell^\perp)$ is dense in $\P (W^\perp)$.
Hence, the set of $\gamma \in W^\perp \subset L_\R$ such that $\gamma^\perp \cap L = W \cap L$ is dense in $L_\R$.
Additionally, since $K$ is a negative-definite sublattice, the pairing $(\cdot,\cdot)$ is indefinite on $W^\perp$.
Therefore, there exists $\gamma \in W^\perp 
\subset L_\R$ with $( \gamma, \gamma )=(\beta,\beta) > 0$ such that
$\gamma^\perp \cap L = W \cap L$, and hence $\kappa^\perp \cap \beta^\perp \cap \gamma^\perp \cap L = K$.

Since $( \kappa,[\beta+\sqrt{-1}\gamma])$ lies in $K\Omega$, Theorem \ref{KT} implies that  there exists a marked polarized generalized 
K3 surface $(X, I, \epsilon, \phi)$ so that  $\phi(\epsilon) =  \kappa$
and $\beta + \sqrt{-1} \gamma \in \phi( H^{2,0}(\Hat X)).$
Moreover, there exists a K\"ahler orbifold form $\zeta \in \Omega^{1,1}(X)$ which defines a closed current on $\Hat X$ in the cohomology class  $\epsilon$. Since $\epsilon \in V_P^+(X)$,
$$
R(X) =\left\{\delta \in H^{1,1}(\Hat X) \cap H^2(\Hat X;\mathbb{Z}) \mid
(\delta,\delta) = -2  \textrm{ and }(\epsilon,\delta)  = 0
 \right\}.
$$
Moreover, since $\beta + \sqrt{-1} \gamma \in \phi(H^{2,0}(\Hat X))$, we have $\Hat\phi(H^{1,1}(\Hat X) ) = \beta^\perp \cap  \gamma^\perp \cap L.$
Since $\kappa=\phi(\epsilon)$ and $\kappa^\perp  \cap \beta^\perp \cap \gamma^\perp \cap L = K$, this implies that
$$ \Hat \phi(R(X)) = \{ d \in K \mid (d, d) = -2 \}.$$
\end{proof}

We are now ready to prove our main proposition.

\begin{proof}[Proof of Proposition~\ref{propKT}]
By Lemma~\ref{roots} below,  
there exist  $\Hat \kappa$, $\beta$, and $d_1,\dots,d_k$ in  $L$ 
such that  
$(\Hat \kappa, \Hat \kappa)=B$, 
$(\beta,\beta)=A$, 
$(\Hat \kappa,\beta) = 0$,
$(\Hat \kappa, d_i) = 1$ 
and 
$(\beta, d_i) = 0$ 
for all $i$, 
and
$(d_i,d_j)=-2 \delta_{ij}$ for all $i$ and $j$;  
moreover, 
$\Hat \kappa$, 
$\beta$, 
and $d_1,\dots,d_k$ 
generate a primitive sublattice.
Let $K$ be the negative definite primitive sublattice generated by $d_1,\dots,d_k$, and
let $\kappa := \Hat\kappa + \frac{1}{2}\sum_{i=1}^k d_i$ be the projection of $\Hat\kappa$ onto $K^\perp \subset L_\R$. Since $(\kappa,\kappa) = B + \frac{1}{2}k > 0$,
Corollary~\ref{corKT} implies that there  exists a marked polarized generalized K3 surface $(X,I,\epsilon, \phi)$ such that $\kappa = \phi(\epsilon)$, $\beta \in \phi(H^{2,0}(\Hat X) \oplus H^{0,2}(\Hat X))$, and $$\Hat \phi(R(X)) = \{d \in K \mid (d,d) = -2 \} = \{\pm d_1,\dots, \pm d_k\}.$$
Moreover, there exists a K\"ahler orbifold form $\zeta \in \Omega^{1,1}(X)$ that defines a closed current on $\Hat X$
in the class $\epsilon$.
Here, $q \colon \Hat X \to X$ is the minimal resolution of $X$ and $\Hat \phi \colon H^2(\Hat X; \Z) \to L$
is any isometry extending $\phi$.

The root system of $X$ is $R(X)=\{ \pm \cE_1,\dots, \pm \cE_k\}$, where $\Hat \phi(\cE_i)=d_i$ for all $i$ and $(\cE_i,\cE_j)=-2\delta_{ij}$ for all $i,j$. By the classification of simple surface singularities (and an analysis of their minimal resolutions), this implies that $X$ has $k$ isolated $\Z_2$ singularities, 
that $q \colon \Hat X \to X$ is the blow up of $X$ at the singular points, and that $\cE_1,\dots,\cE_k \in H^{1,1}(\Hat X) \cap H^2(\Hat X; \Z)$ are the Poincar\'e duals to the exceptional divisors. 

Since $\beta \in  \phi(H^{2,0}(\Hat X) \oplus H^{0,2}(\Hat X)) \cap L$,
there exists a holomorphic form $\Hat\mu \in \Omega^{2,0}(\Hat X)$
such that
$\textrm{Re} (\phi([\Hat\mu]))= \beta$,
where $\textrm{Re}(\cdot)$ denotes the real part.
Since $(\beta,\beta) > 0$ and $\Hat X$ is a  K3 surface,
this implies that $\Hat \mu$  is nonvanishing.
By Lemma \ref{l38},
there exists a nonvanishing holomorphic orbifold form $\mu \in \Omega^{2,0}(X)$ such that $q^*([\mu])=[\Hat\mu]$. 
Thus,
$\nu := \textrm{Re}(\mu) \in  \Omega^{2,0}(X)\oplus \Omega^{0,2}(X)$ is 
a real symplectic orbifold form such that $\phi(q^*([\nu])) =\beta$. Since also $q^*(\zeta)=\epsilon$ by Lemma \ref{current}, the orbifold forms $[\zeta]$ and $\nu$ have the desired properties.

Define a class $\Hat \epsilon := \Hat \phi^{-1} (\Hat \kappa)=\epsilon - \frac{1}{2} \sum_{i=1}^k \cE_i \in H^{1,1}(X) \cap H^2(\Hat X;\mathbb{Z})$.
Since $\Hat X$ is a closed K\"ahler manifold, the Lefschetz theorem on $(1,1)$-classes
imples that there exists a holomorphic $\C^\times$ bundle over $\Hat X$
with Euler class $- \Hat \epsilon$.
Therefore, since $q^*([\zeta]) = \epsilon=\Hat \epsilon+\frac{1}{2}\sum_{i=1}^k \cE_i$, and since $( \Hat \epsilon, \cE_i ) = 1$ for all $i$, Lemma \ref{l39} implies that 
there exists a holomorphic $\C^\times$ bundle $\cL$ over $X$  with Euler class $-[\zeta]$ so that $\cL$ is a manifold. \end{proof}

We conclude this section with an easy technical lemma.

\begin{Lemma}\labell{roots}
Fix  $k \in \{0,1,\dots ,9\}$ and even integers $B$  and $A$.
There exist  $\Hat \kappa$, $\beta$, and $d_1,\dots,d_k$ in  $L$ 
such that  
$(\Hat \kappa, \Hat \kappa)=B$, 
$(\beta,\beta)=A$, 
$(\Hat \kappa,\beta) = 0$,
$(\Hat \kappa, d_i) = 1$ 
and 
$(\beta, d_i) = 0$ 
for all $i$, 
and
$(d_i,d_j)=-2 \delta_{ij}$ for all $i$ and $j$;  
moreover, 
$\Hat \kappa$, 
$\beta$, 
and $d_1,\dots,d_k$ 
generate a primitive sublattice.
\end{Lemma}

\begin{proof} 
We will give an explicit construction for the case $k=9$; the other cases follow immediately.

Let $H$ denote the unimodular hyperbolic plane, i.e.,
the even unimodular lattice of signature $(1,1)$; let $E_8$ denote 
the  positive definite even unimodular lattice of rank $8$.
Consider the following orthogonal decomposition of the K3 lattice:
$$L = H \oplus H \oplus H  \oplus -E_8 \oplus -E_8.$$
For each $i \in \{1,2,3\}$, choose a basis
$x_{i,1}$ and $x_{i,2}$ for the $i$'th summand of $H$ in the decomposition above such that
$$(x_{i,1},x_{i,1}) = (x_{i,2},x_{i,2}) = 0 \quad \mbox{and} \quad (x_{i,1},x_{i,2}) = 1.$$
Similarly, for each $i \in \{1,2\}$, choose a basis
$y_{i,1},y_{i,2},\dots, y_{i,8}$ for the $i$'th summand of $-E_8$ such that
$$(y_{i,j},y_{i,k}) = 
\begin{cases}
1 & \mbox{ if } \{j,k\} \in \{ \{1,2\}, \{2,3\}, \dots, \{6,7\}\} \cup \{\{5,8\}\}, \\
-2 & \mbox{ if } j = k, \mbox{ and} \\
0 & \mbox{ otherwise}.
\end{cases}
$$

Let $\Hat \kappa :=  x_{1,1} + \frac{1}{2} B x_{1,2}$, 
$\beta := x_{2,1} + \frac{1}{2}A x_{2,2}$, and 
$$d_i :=
\begin{cases}
x_{3,1}-x_{3,2}+x_{1,2} & \mbox{ if } i=1, \\
y_{1,2i-3}+x_{1,2} & \mbox{ if } i \in \{2,3,4,5\}, \\
y_{2,2i-11}+x_{1,2} & \mbox{ if } i \in \{6,7,8,9\}.
\end{cases}
$$
To see that $\Hat \kappa$, $\beta$, and $d_1,\dots,d_k$ generate a primitive sublattice,
consider the dual basis for the subspace that they span:
$\Hat \kappa' :=   x_{1,2}$, $\beta' := x_{2,2}$,
$d_1' := x_{3,2}$
$d_2' := y_{1,2} - y_{1,4} + y_{1,8}$, 
$d_3' :=  y_{1,4} - y_{1,8}$, 
$d_4' :=  y_{1,8}$, 
$d_5' := y_{1,6} - y_{1,8}$, 
$d_6' := y_{2,2} - y_{2,4} + y_{2,8}$, 
$d_7' :=  y_{2,4} - y_{2,8}$, 
$d_8' :=  y_{2,8}$, and
$d_9' := y_{2,6} - y_{2,8}$.
\end{proof}

\section{Locally free Hamiltonian actions }
\labell{lfree}

In this section, we consider locally free Hamiltonian circle actions on symplectic manifolds
with reduced spaces diffeomorphic to generalized K3 surfaces with
isolated $\Z_2$ singularities.
We do not give a general classification, but simply construct the needed examples.
These correspond to the portions of the manifolds 
described in the introduction that are locally free and have convex Duistermaat-Heckman functions, e.g., the part of the manifold in
Theorem \ref{theorem10} that lies over $(1,9) \subset \R/10\Z$.
Moreover, we endow each example with two congenial complex structures that
we use to ``add fixed points" in the next section.

\begin{Proposition}\labell{plf2}
Fix $k \in \{0, 1, \dots, 9\}$, a positive even integer $A$,
and an even integer $B > -\frac{1}{2}k$.
There exists 
a manifold $\breve M$ with a locally free $\C^\times$ action, an $S^1 \subset \C^\times$
invariant symplectic form $\breve \omega \in \Omega^2(\breve M)$, 
with a proper moment map $\breve\Psi \colon \breve M \to \R$,
a $\C^\times$ invariant map $\breve \pi$ from $\breve M$ to
a generalized K3 surface $(X,I)$ with $k$ isolated $\Z_2$ singularities,
and two complex
structures $\breve J_+$ and $\breve J_-$ on $\breve M$
that commute with the $\C^\times$ action\footnote{More precisely, we insist that the map  $\C^\times \times \breve M \to \breve M$ induced by the action is holomorphic with respect to each complex structure.}, so that the following hold:
\begin{enumerate}
\item For all $t \in \R$, the map $\breve \pi$ induces a symplectomorphism from $\breve M \modt$ to $X$ 
with  symplectic form $\breve \sigma_t \in \Omega^2(X)$;  moreover,
$(\breve \sigma_t,\breve \sigma_t)= A + (B + \frac{1}{2}k)  t^2$ and
$q^*([\breve \sigma_{2 \ell}])$ is primitive
for all $\ell \in \Z$.
Here,  $q \colon \Hat X \to X$ is the minimal resolution of $X$.
\item $\breve \pi$ induces a biholomorphism  from $(\breve M,\breve J_+)/\C^\times$ to $(X,I)$, and an
anti-biholomorphism from $(\breve M,\breve J_-)/\C^\times$ to $(X,I)$.
\item $\breve \omega(\breve\xi, \breve J_\pm(\breve\xi)) > 0$, where $\breve \xi \in \chi(\breve M)$ generates the $S^1$ action.
\item $\breve J_\pm$ tames $\breve \omega$ on the preimage $\breve \Psi\inv\big(\pm (0,\infty)\big)$.
\end{enumerate}
\end{Proposition}

Proposition~\ref{plf2} follows immediately from Proposition~\ref{propKT} and Proposition~\ref{p:two} below.
To begin the proof of the latter proposition, we first confirm that a standard property of holomorphic bundles over closed K\"ahler manifolds also holds for orbifolds. (See Definition~\ref{d37}.)

\begin{Lemma}\labell{deldelbar}
Let $(X,I)$ be a closed complex orbifold that admits a K\"ahler orbifold form.
Let $p \colon \cL \to X$ be a holomorphic $\mathbb{C}^\times$ bundle with Euler class $-[\zeta]$,
where $\zeta \in \Omega^{1,1}(X)$ is a real orbifold form.
Then there exists a hermitian metric on the associated line bundle containing $\cL$ with Chern curvature $2 \pi \sqrt{-1} \, \zeta$.
\end{Lemma}

\begin{proof}
Fix any hermitian metric $h'$ on the associated line bundle containing $\cL$.
Let $\zeta' \in \Omega^{1,1}(X)$ be the real orbifold form
with local expression
$$\textstyle \frac{\sqrt{-1}}{2 \pi} \del \delbar s^*(\Psi')$$
for every local holomorphic section $s$ of $\cL$,
where $\Psi'(m) := \ln |m|^2_{h'}$ for all $m \in \cL$.
Since $-\zeta$ and $-\zeta'$ both represent the Euler class of $\cL$, the difference $\zeta - \zeta'$ is $d$-exact.
Moreover, since $X$ is a closed complex orbifold that admits a  K\"ahler orbifold form, $X$ satisfies the $\del \delbar$-lemma by \cite[Lemma 5.4]{BBFMT}; (see also \cite{B}).
Hence, $X$  also satisfies the $d d_c$-lemma.
Therefore, there exists a function $f \colon X \to \R$ such that
$$\textstyle \frac{\sqrt{-1}}{2 \pi} \del \delbar f = \zeta - \zeta'.$$
Let $h$ be  the hermitian metric on the associated line bundle containing $\cL$ determined by $h = e^{f} h'$.
Then 
$$\textstyle \zeta =  \frac{\sqrt{-1}}{2 \pi} \del \delbar s^*(\Psi)$$
for every local holomorphic section $s$ of $\cL$, where $\Psi(m) := \ln |m|_h^2 = \Psi'(m) +  f(p(m))$ for all $m \in \cL$.
\end{proof}

Next we show how, in certain cases, we can use forms on a complex orbifold to construct a symplectic form
on a holomorphic $\C^\times$ bundle over that orbifold
so that the natural complex structure on the bundle has the desired properties.

\begin{Lemma}\label{p:one}
Let $(X,I)$ be a closed complex orbifold, $\zeta \in \Omega^{1,1}(X)$ a K\"ahler orbifold form, and $\nu \in \Omega^{2,0}(X) \oplus \Omega^{0,2}(X)$ a real symplectic orbifold form.
Let $p \colon \cL \to X$ be a holomorphic $\C^\times$ bundle over 
$X$ with Euler class $-[\zeta]$; let $J$ be the natural complex structure on $\cL$. 
Then there exists a hermitian metric on the associated line bundle containing $\cL$ 
such that $\omega := p^*(\nu) + \frac{\sqrt{-1}}{4 \pi} \del \delbar \Psi^2 \in \Omega^2(\cL)$ is an $S^1 \subset \C^\times$ invariant  symplectic
form with moment map $\Psi$, where $\Psi(m) := \ln |m|^2$ for all $m \in \cL$.
Moreover, the following hold:
\begin{enumerate}
\item For all $t \in \R$, the map $p$ induces a symplectomorphism from $\cL \modt $ to $X$ 
with the symplectic form $\nu + t \zeta \in \Omega^2(X)$.
\item $p$ induces a biholomorphism from $(\cL,J)/C^\times$ to $(X,I)$.
\item $\omega(\xi, J(\xi)) > 0$, where $\xi \in \chi(\cL)$ generates the $S^1$ action.
\item $J$ tames $\omega$ on $\Psi\inv\big((0,\infty) \big)$.
\end{enumerate}
\end{Lemma}

\begin{proof}
By Lemma~\ref{deldelbar}, there exists
a hermitian metric on the associated line bundle containing $\cL$ such that
$$\textstyle \zeta = \frac{\sqrt{-1}}{2 \pi} \del \delbar s^*(\Psi)$$
for every local holomorphic section $s$ of $\cL$, where $\Psi(m) := \ln |m|^2$ for all  $m \in \cL$. Equivalently, $p^*(\zeta)=\frac{\sqrt{-1}}{2 \pi} \del \delbar \Psi$. Define an
$S^1$ invariant closed real form $\eta \in \Omega^{1,1}(\cL)$
by 
$$\textstyle \eta :=   
\frac{\sqrt{-1}}{4 \pi} \del \delbar \Psi^2
= \frac{\sqrt{-1}}{2 \pi} \left(  \Psi \del \delbar \Psi +  \del \Psi \wedge \delbar \Psi  \right) 
.$$
By a straightforward calculation in coordinates $d\Psi \neq 0$; moreover, 
$\iota_\xi \del \Psi=2 \pi i$, $\iota_\xi \delbar \Psi=-2 \pi i$, and $\iota_\xi \del \delbar \Psi=0.$
Therefore,
$$\textstyle \iota_\xi \eta=- \delbar \Psi - \del \Psi = - d \Psi.$$
By another straightforward calculation, this implies that
$$\textstyle \eta(\xi, J(\xi)) = 4 \pi > 0.$$
Moreover, since $\del \Psi \wedge \delbar \Psi =
\frac{1}{2}  d \Psi \wedge  (\delbar \Psi - \del \Psi),$
$$\textstyle \eta|_{\Psi\inv(t)} = \frac{\sqrt{-1}}{2 \pi} t \del \delbar \Psi= t p^*(\zeta)$$
 for all $t \in \R$. 

Fix any $m \in \cL$ with $t := \Psi(m)  > 0$.
Since $\eta(\xi,J(\xi)) > 0$, $T_m\cL$ is the direct sum of the subspace
$$V:= \{ v \in T_m \cL \mid \eta(v,\xi) = \eta(v,J(\xi)) = 0 \} \subset T_m \cL$$
and the subspace spanned by $\xi$ and $J(\xi)$.
Moreover, since $\eta$ is a $(1,1)$-form, the pullback $J^*(\eta)$ is equal to $\eta$.
Hence, $\eta(J(v),\xi)=-\eta(v,J(\xi))=0$ and $\eta(J(v),J(\xi))=\eta(v,\xi)=0$ for all $v \in V$.
Thus, $V$ is $J$-invariant and the map $p_* \colon T_m \cL \to T_{p(m)} X$ restricts to an isomorphism of complex vector spaces
$$\textstyle V  \to T_{p(m)} X.$$
Since $\iota_\xi\eta=-d\Psi$, $V$ is a subspace of $T_m\Psi^{-1}(t)$.
Since $\eta|_{\Psi\inv(t)} = t p^*(\zeta)$,
this implies that $\eta(v,J(v)) = t \zeta(p_*(v),I(p_*(v)))$ for all $ v \in V$. 
Furthermore, since  $t > 0$ and $\zeta$
is a K\"ahler orbifold form on $(X,I)$, the
latter quantity is positive  if $v \neq 0$.
Hence, $\eta(v,J(v)) > 0$ for all non-zero $v \in V$.
Since $\eta(\xi,J(\xi))>0$, this implies that $\eta(v,J(v)) > 0$ for all nonzero $v \in T_m \cL$. 
Thus, $\eta$ is a K\"ahler form on $\Psi^{-1}((0,\infty))$.

Finally, since $\nu \in \Omega^{2,0}(X) \oplus \Omega^{0,2}(X)$ and $\zeta$ is a K\"ahler orbifold form, if $t \neq 0$ then $(\nu + t \zeta)(w,I(w)) = t \zeta(w,I(w)) \neq 0$ for every non-zero $w \in TX$.
Since $\nu$ is symplectic, this implies that
$\nu + t \zeta$ is  symplectic for all $t \in \R$.
Define an
$S^1$ invariant closed $2$-form $\omega \in \Omega^2(\cL)$
by $\omega :=  p^*(\nu) + \eta$. 
By the results in the first paragraph, $\Psi$ is regular, $\iota_\xi \omega = - d \Psi$, and $\omega|_{\Psi\inv(t)} =  p^*(\nu+ t \zeta)$ for
all $t \in \R$. Since $\nu + t \zeta$ is symplectic for all $t$,
this implies that $\omega$ is symplectic, $\Psi$ is a moment
map for the $S^1$ action on $\cL$, and $p$ induces a symplectomorphism
from  
 $\cL \modt$ to $X$ with the symplectic form $\nu + t \zeta$ for all $t \in \R$.
Finally, since $p^* (\nu) \in \Omega^{2,0}(\cL) \oplus \Omega^{0,2}(\cL)$,
$\omega(v,Jv) = \eta(v,Jv)$ for all $v \in T\cL$.
Therefore, by the preceding paragraphs, $\omega(\xi,J(\xi)) > 0$
and $J$ tames $\omega$ on $\Psi^{-1}((0,\infty))$. \end{proof}

Finally, we show how in the situation of Proposition~\ref{p:one} we can construct another complex structure that also has the desired properties.

\begin{Proposition}\label{p:two}
Let $(X,I)$ be a closed complex orbifold, $\zeta \in \Omega^{1,1}(X)$ be a K\"ahler orbifold form, and $\nu \in \Omega^{2,0}(X) \oplus \Omega^{0,2}(X)$ be a real symplectic orbifold form.
Let $p \colon \cL \to X$ be a holomorphic $\C^\times$ bundle with Euler class $-[\zeta]$.
Then there exist an $S^1 \subset \C^\times$ invariant symplectic form $\omega \in \Omega^2(\cL)$ with a proper moment map $\Psi \colon \cL \to \R$, and complex structures $J$ and $J'$ on $\cL$ that
commute with the $\C^\times$ action, so that the following hold:
\begin{enumerate}
\item For each $t \in \R$, the map $p$ induces a symplectomorphism from $\cL \modt$ to $X$ 
with the symplectic form $\nu + t \zeta \in \Omega^2(X)$.
\item $p$ induces a biholomorphism  from $(\cL,J)/\C^\times$ to $(X,I)$, and an
anti-biholomorphism from $(\cL,J')/\C^\times$ to $(X,I)$.
\item $\omega(\xi, J(\xi)) > 0$ and $\omega(\xi, J'(\xi)) > 0$, where $\xi \in \chi(\cL)$ generates the $S^1$ action.
\item $J$ and $J'$ tame $\omega$ on the preimages $\Psi\inv\big((0,\infty)\big)$ and $\Psi\inv\big((-\infty,0)\big)$, respectively.
\end{enumerate}
\end{Proposition}

\begin{proof}
Let $J$ be the natural complex structure on $\cL$.
By Lemma~\ref{p:one}, there exists
a  hermitian metric on the associated line bundle containing $\cL$ such that
$$\textstyle \omega := p^*(\nu) + \frac{\sqrt{-1}}{4 \pi} \del \delbar \Psi^2 = p^*(\nu) + \frac{\sqrt{-1}}{2\pi}(\Psi \del \delbar \Psi + \del \Psi \wedge \delbar \Psi) \in \Omega^2(\cL)$$
 is an $S^1$ invariant  symplectic form with moment map  $\Psi$, where $\Psi(m) := \ln |m|^2$ for all $m \in \cL$. (Here, and elsewhere in this proof, we take the Dolbeault operators with respect to $J$.) Moreover, the following hold:
\begin{enumerate}
\item For all $t \in \R$, the map $p$ induces a symplectomorphism from $\cL \modt $ to $X$ 
with the symplectic form $\nu + t \zeta \in \Omega^2(X)$.
\item $p$ induces a biholomorphism from $(\cL,J)/\C^\times$ to $(X,I)$.
\item $\omega(\xi, J(\xi)) > 0$, where $\xi \in \chi(\cL)$ generates the $S^1$ action.
\item $J$ tames $\omega$ on $\Psi\inv\big((0,\infty) \big)$.
\end{enumerate}

Define a diffeomorphism $\Lambda \colon \cL \to \cL$ by
$\Lambda(m) := \frac{m}{|m|^2} $ for all $m\in \cL$; note that $p \circ \Lambda=p$.
Let $J'$ be the complex structure on $\cL$ that is conjugate to the pull-back of $J$ by $\Lambda$, i.e., 
$$\textstyle J'(v)=-\Lambda_*(J(\Lambda_*^{-1}(v)))$$
for all $v \in T\cL$.

The map $\Lambda$ intertwines the standard $\mathbb{C}^\times$-action and the conjugate inverse action on $\cL$, i.e.,
$\Lambda( \lambda \cdot m) = \overline{\lambda}^{-1} \cdot \Lambda(m)$ for all $\lambda \in \C^\times$ and $m \in \cL$.
Hence, since the standard $\mathbb{C}^\times$-action on $\cL$ is $J$-holomorphic and the map sending $\lambda \in \mathbb{C}^\times$ to $\overline{\lambda}^{-1}$ is anti-holomorphic, the standard $\mathbb{C}^\times$-action is also $J'$-holomorphic.
Moreover, since $p$ induces a biholomorphism from $(\cL,J)/\C^\times$ to $(X,I)$, it induces an anti-biholomorphism from $(\cL,J')/\C^\times$ to $(X,I)$.

By a straightforward calculation in coordinates, we have $\Lambda^*(\Psi) = -\Psi$, $\Lambda^*(\del \Psi)=-\delbar \Psi$, $\Lambda^*(\delbar \Psi)=-\del \Psi$, and $\Lambda^*(\del \delbar \Psi)=\del \delbar \Psi$. Therefore,
$$\textstyle \Lambda^*(\omega)=2p^*(\nu)-\omega.$$
Because $p$ is anti-holomorphic with respect to $J'$ and $\nu \in \Omega^{2,0}(X) \oplus \Omega^{0,2}(X)$,
$$\textstyle p^*(\nu)(v,J'(v))=\nu(p_*(v),p_*(J'(v)))=-\nu(p_*(v),I(p_*(v)))=0$$
for all $v \in T \cL$. The three displayed equations above imply that $\omega(v,J'(v))=\omega(\Lambda_*^{-1}(v),J(\Lambda_*^{-1}(v)))$
for all $v \in T\cL$. In particular, since $\Lambda$ is $S^1$-equivariant, $\omega(\xi,J'(\xi))=\omega(\xi,J(\xi))>0$.
Moreover, since $J$ tames $\omega$ on $\Psi^{-1}((0,\infty))$ and $\Psi^{-1}((-\infty,0))=\Lambda(\Psi^{-1}((0,\infty)))$,
this implies that $J'$ tames $\omega$ on $\Psi^{-1}((-\infty,0))$.
\end{proof}

\section{Circle actions with fixed points}
\labell{fixed}

In this section, we finally consider Hamiltonian circle actions on symplectic manifolds with
fixed points. However, each regular reduced space is still symplectomorphic to a generalized
K3 surface with isolated $\Z_2$ singularities.
As in the previous section, we don't give a complete classification
but simply construct the examples that we need, corresponding to 
portions of the manifolds described in the introduction
that contain all the points with non-trivial stabilizers, e.g., the part of the manifold in Theorem~\ref{theorem10} that lies
over $(1- \tepsilon, 9 + \tepsilon) \subset \R/10\Z$,
where  $\tepsilon >0$.

\begin{Proposition}\labell{existsfixed2}
Fix  $k \in \{1, \dots, 9\}$, a positive even integer $A$,
and an even integer $B > - \frac{1}{2} k$.
There exists a 
circle action on a symplectic manifold $(\Tilde{M},\Tilde{\omega})$ 
with
a proper moment map $\Tilde\Psi \colon \Tilde{M} \to ( -4-\Tilde{\epsilon}, 
 4 + \Tilde{\epsilon})$, where $\Tilde \epsilon > 0$,  with the following properties:
The level sets  $\Tilde{\Psi}\inv( \pm 4)$ each contain exactly $k$ fixed
points
with weights $\pm \{-2,1,1\}$; otherwise, the action is locally free.
The Duistermaat-Heckman function of $\tM$ is
$$\Tilde \varphi(t) = \begin{cases} 
A - 8k - 4kt + Bt^2  &  t \in ( - 4 - \Tilde{\epsilon},  -4]  \\
A +(B + \frac{1}{2}k) t^2 & t \in [ -4,   4] \\
A  - 8k + 4kt + Bt^2  &  t \in [ 4,  4 + \Tilde{\epsilon}). \\
\end{cases}
$$
The reduced space $\Tilde{M} \modt $ is
diffeomorphic to a generalized K3 surface with $k$ isolated $\Z_2$ singularities  for all $t \in (-4,4)$,
and symplectomorphic to a tame K3 surface   
$(\Hat{X},\Hat{I},(\Hat{\sigma}_\pm)_t)$, 
where $(\Hat{\sigma}_\pm)_t \in \Omega^2(\Hat{X})$ satisfies
$\big[(\Hat{\sigma}_\pm)_t\big] = \Hat{\kappa}_\pm -  t \Hat{\eta}_\pm$,
for all $t \in  \pm (4, 4 +  \Tilde{\epsilon})$; moreover,
$\Hat{\kappa}_\pm,\Hat{\eta}_\pm$
induce a primitive embedding $\Z^2 \hookrightarrow H^2(\Hat{X};\Z)$.
\end{Proposition}

This proposition will follow from Proposition~\ref{propKT} and the following more general construction.

\begin{Proposition} \label{p62}
Let $(X,I)$ be a closed $2$-dimensional complex orbifold with $k$ isolated $\mathbb{Z}_2$ singularities, $\zeta \in \Omega^{1,1}(X)$ be a K\"ahler orbifold form, and $\nu \in \Omega^{2,0}(X) \oplus \Omega^{0,2}(X)$ be a real symplectic orbifold form.
Let $p \colon \cL \to X$ be a holomorphic $\C^\times$ bundle with Euler class $-[\zeta]$ so that $\cL$ is a manifold. Fix $c_+,c_->0$.

There exists a circle action on a symplectic manifold $(\tM, \tomega)$ with a proper moment map $\widetilde \Psi \colon \tM \to (-c_- - \widetilde \epsilon, c_+ + \widetilde \epsilon)$, where $\widetilde \epsilon > 0$, with the following properties: The level sets $\widetilde \Psi^{-1}( \pm c_\pm)$
each contain exactly $k$ fixed points with weights $\pm \{-2,1,1\}$;
otherwise, the action is locally free. 
For each $t \in (-c_-,c_+)$, the reduced space $\tM \modt$ is symplectomorphic to
$(X,\tsigma_t)$, where $[\tsigma_t] = [\nu + t \zeta]$.
For each $t \in \pm (c_\pm, c_\pm + \widetilde \epsilon)$, the reduced space $\tM \modt$ is symplectomorphic to 
$(\Hat X, (\Hat \sigma_\pm)_t)$, where $\Hat I$ tames $(\Hat \sigma_\pm)_t$ and
$[(\Hat \sigma_\pm)_t] = q^*([\nu+t \zeta]) + (c_\pm \mp t)\sum \cE_i / 2$; 
moreover, $-q^*([\zeta]) \pm \sum \cE_i / 2 \in H^2(\Hat X;\mathbb{Z})$.
Here, $q\colon (\Hat X, \Hat I) \to (X,I)$ is the blow up of $X$ at the singular points
and the $\cE_1,\dots,\cE_k  \in H^{1,1}(\Hat X;\R)$ are the Poincar\'e duals to the exceptional divisors.
\end{Proposition}

\begin{proof}
Let $\omega \in \Omega^2(\cL)$ be the $S^1 \subset \C^\times$ invariant symplectic form  with a proper moment map $\Psi \colon \cL \to \R$, and let $J$ and $J'$ be the complex structures on $\cL$ that
commute with the $\C^\times$ action, described in Proposition~\ref{p:two}.

We will prove the following claim: There exists $\tepsilon > 0$,
and circle actions on symplectic manifolds $(\tM_\pm,\tomega_\pm)$ with
proper moment maps $\tPsi_\pm \colon \tM_\pm \to \R$ with the following properties:
\begin{enumerate}
\item [(A)] The preimage $\tPsi_\pm^{-1}(-\tepsilon,\tepsilon) \subset \tM_\pm$ is equivariantly symplectomorphic
to $\Psi^{-1}(-\tepsilon,\tepsilon) \subset \cL$.
\item [(B)] The preimage
$\tPsi_\pm^{-1}\big(\pm (-\infty,c_\pm + \tepsilon)\big)$  contains 
exactly $k$ fixed points;  each lies in
$\tPsi_\pm^{-1}(\pm c_\pm)$  and has 
weights $\pm\{-2,1,1\}$.  
\item [(C)] For each $t \in \pm(-\infty,c_\pm)$, 
the reduced space $\tM_\pm \modt$ is symplectomorphic to
$(X,(\tsigma_\pm)_t)$, where $[(\tsigma_\pm)_t] = [\nu + t \zeta]$.
\item [(D)] For each $t \in \pm(-\infty, c_\pm + \widetilde \epsilon)$, 
the reduced space $\tM_\pm \modt$ is symplectomorphic to 
$(\Hat X, (\Hat \sigma_\pm)_t)$, where $\Hat I$ tames $(\Hat \sigma_\pm)_t$ and
$[(\Hat \sigma_\pm)_t] = q^*([\nu+t \zeta]) + (c_\pm \mp t)\sum \cE_i / 2$; 
moreover, $-q^*([\zeta]) \pm \sum \cE_i / 2 \in H^2(\Hat X;\mathbb{Z})$.
\end{enumerate}

Once we have proved this claim, we can conclude that $\tPsi_+^{-1}(-\tepsilon,\tepsilon) \subset \tM_+$
and $\tPsi_-^{-1}(-\tepsilon,\tepsilon) \subset \tM_-$ are equivariantly symplectomorphic.
Then we can use this symplectomorphism
to glue together $\tPsi_+^{-1}(-\tepsilon,c_+ + \tepsilon) \subset \tM_+$ and 
$\tPsi_-^{-1}(-c_- -\tepsilon,\tepsilon) \subset \tM_-$  to construct a circle action on
a symplectic manifold $(\tM,\tomega)$ with a proper moment map $\tPsi \colon \tM \to (-c_- - \tepsilon,
c_+ + \tepsilon)$ satisfying the desired properties.

We begin by constructing $\tM_+$.
By the first claim in Proposition~\ref{p:two},
for each $t \in \R$, the map $p$ induces a symplectomorphism from $\cL \modt$ to $X$ 
with the symplectic form $\nu + t \zeta \in \Omega^2(X)$.

The second claim of Proposition~\ref{p:two}  implies that
 $p$  also induces a biholomorphism from $(\cL,J)/\C^\times$ to $(X,I)$;
hence,  $U_t  :=  \C^\times \cdot \Psi^{-1}(t)$ is equal to $\cL$ for all $t \in \R$.
Moreover,  by the third claim, $\omega(\xi, J(\xi)) > 0$, where $\xi \in \chi(\cL)$ generates the $S^1$ action.
Hence, by \cite[Proposition 3.1]{TWa}, $J$ induces a complex structure
on the reduced space $\cL \modt$ so that the natural map $\cL \modt \to (U_t, J)/\C^\times$ is a biholomorphism. 
Therefore $p$ induces a biholomorphic symplectomorphism from $\cL \modt$ to $(X,I,\nu + t \zeta)$.

The last claim of Proposition~\ref{p:two} implies that $J$ tames $\omega$ on
the preimage $\Psi^{-1}\big((0,\infty)\big)$.
Additionally, since the reduced space $\cL \modc$ 
has $k$ isolated $\Z_2$ singularities, 
the $S^1$ action on $\Psi^{-1}(c_+)$ is free except for $k$ orbits with stabilizer $\Z_2$.
Therefore, we can apply \cite[Proposition 6.1]{TWa} (centered at $c_+$ instead of at $0$) to $(\cL, J)$
to find 
$\epsilon \in (0,2c_+/3)$, a complex manifold
$(\tM_+,\tJ_+)$ with a holomorphic $\C^\times$ action, an $S^1 \subset \C^\times$ invariant  symplectic form $\tomega_+ \in \Omega^2(\tM_+)$
satisfying $\tomega_+(\xi_{\tM_+}, \tJ_+(\xi_{\tM_+})) > 0$ on $\tM_+ \ssetminus \tM_+^{S^1}$, where $\xi_{\tM_+}$ generates the $S^1$ action,
and a proper moment map $\tPsi_+ \colon \tM_+ \to \R$ so that the following hold:
\begin{enumerate}
\item  $\Tilde\Psi_+\inv\big((c_+-\epsilon,c_+]\big)$
contains exactly $k$ fixed points; each lies in $\tPsi_+^{-1}(c_+)$ and has
weights $\{-2,1,1\}$.
\item There's an $S^1$ equivariant symplectomorphism
from $\Tilde\Psi_+\inv \big((-\infty,c_+ -\epsilon/2)\big)$ to $\Psi\inv \big((-\infty,c_+ -\epsilon/2)\big)$
that induces a biholomorphism from $\tM_+ \modt$ to $\cL \modt$ for all regular $t \in (-\infty, c_+ -\epsilon/2)$.
\item $\tomega_+$ is tamed on $\tPsi_+^{-1} \big((c_+ - \epsilon, c_+ + \epsilon)\big).$
\end{enumerate}

By Property (2) above, Claim (A) holds for all $\tepsilon \in (0,\epsilon)$.
Moreover, since the $S^1$ action on $\cL$ is locally free,
Properties (1) and (2) combine to show that there exists
$\tepsilon \in (0,\epsilon)$ so that Claim (B) holds.
Hence, by the paragraph above, Property (2) implies that 
$\tM_+ \modt$ is biholomorphically symplectomorphic to $(X, I, \nu + t \zeta)$
for each $t \in (-\infty, \tepsilon)$.
Moreover, $-[\zeta]$ is the Euler class of $\tPsi_{+}^{-1} (t) \to X$. 

Next, apply  \cite[Proposition 7.9]{TWa} (centered at $c_+$ instead of $0$) to 
$\tM_+$ with  $(a_-,a_+) = (-\infty,c_+ + \tepsilon)$.
By the preceding paragraph, we may assume that the complex orbifold and classes appearing in \cite[Proposition 7.9]{TWa}
are $(X,I)$ and $[\nu], -[\zeta] \in H^2(X;\R)$.  
(To see this take any $t < \tepsilon$.)
More precisely, the following hold:

\begin{itemize}
\item 
For each $t \in (-\infty,c_+)$ 
the   reduced space $\tM_+ \modt$ is  biholomorphically symplectomorphic to 
$(X,I,(\tsigma_+)_t)$, where $(\tsigma_+)_t \in \Omega^2(X)$ satisfies $[(\tsigma_+)_t] = [\nu + t \zeta]$. 
\item 
For each $t \in (c_+,c_+ + \tepsilon)$, 
the reduced space $\tM_+ \modt$ is biholomorphically symplectomorphic to 
$(\Hat X,\Hat I,(\Hat \sigma_+)_t)$, where $(\Hat \sigma_+)_t =q^*([\nu+t \zeta]) + (c_+ - t)\sum \cE_i / 2$;
moreover, $-q^*([\zeta]) + \sum \cE_i / 2 \in H^2(\Hat X;\mathbb{Z})$ because it is the Euler class of $\tPsi^{-1}_+(t) \to \Hat X$.
\end{itemize}

Claim (C) follows immediately.
Additionally, by Property (3) above and \cite[Proposition 3.1]{TWa}, 
$\Hat I$ tames $(\Hat \sigma_+)_t$ for all $t \in (c_+,c_++\tepsilon)$. 
So Claim (D) holds.

The construction of $\tM_-$ is the mirror image to that of $\tM_+$.
Let $I'$ be the conjugate complex structure on $X$.
By a slight modification of the preceding argument, if we equip $\cL$ with the complex structure
$J'$, then $p$ induces a biholomorphic
symplectomorphism from $\cL \modt$ to $(X,I', \nu + t \zeta)$. 
We then need the following variant of \cite[Proposition 6.1]{TWa}, which still holds by \cite[Remark 6.3]{TWa}:
Replace $(-\epsilon,0]$ by $[0,\epsilon)$ and $\{-2,1,\dots,1\}$ by $\{2, -1,\dots,-1\}$ in Claim (1) of that proposition;
and replace each $(-\infty, -\epsilon/2)$ by $(\epsilon/2,\infty)$ in Claim (2).
We apply this new proposition (centered at $-c_-$) to $(\cL, J')$.
Similarly, we apply the following variant of \cite[Proposition 7.9]{TWa} (centered at $-c_-$) to $\tM_-$ with $(a_-,a_+) = (-c_- - \tepsilon, \infty)$.
(which still holds by \cite[Remark 7.11]{TWa}.)
Replace $\{-2,1\dots,1\}$ by $\{2,-1,\dots,-1\}$; replace each
$(0,a_+)$ by $(a_-,0)$ and each $(a_-,0)$ by $(0,a_+)$; and replace each
$\sum \mathcal{E}_i$ by $-\sum \mathcal{E}_i$.  
In this case, for all $t \in (-c_--\epsilon, -c_-)$, the reduced space $\tM_- \modt$ is 
biholomorphically symplectomorphic to $(\Hat X,\Hat I,  (\Hat \sigma_-)_t)$,
where $(\Hat \sigma_-)_t =q^*([\nu+t \zeta]) + (c_- + t)\sum \cE_i / 2$;
moreover, $-q^*([\zeta]) - \sum \cE_i / 2 \in H^2(\Hat X;\mathbb{Z})$.
Otherwise, the argument is strictly analogous to the argument for $\tM_+$. \end{proof}

We are now ready prove our main proposition.

\begin{proof}[Proof of Proposition~\ref{existsfixed2}]
Together, Proposition~\ref{propKT} and Proposition~\ref{p62} (with $c_\pm = 4$) immediately imply this proposition,
except that we still need to prove that $\Hat \kappa_\pm, \Hat \eta_\pm$ induce
a primitive embedding $\mathbb{Z}^2 \hookrightarrow H^2(\Hat X;\mathbb{Z})$,
where $\Hat\kappa_\pm:=q^*([\nu]) + 2 \sum \cE_i \in H^2(\Hat X; \mathbb{R})$, $\Hat \eta_\pm:=-q^*([\zeta]) \pm \sum \cE_i / 2 \in H^2(\Hat X;\mathbb{Z})$, and
$\nu, \zeta \in \Omega^2(X)$.
Here, $q\colon \Hat X \to X$ is the minimal resolution and the $\cE_1,\dots,\cE_k  \in H^{1,1}(\Hat X;\R)$ are the Poincar\'e duals to the exceptional divisors.
Additionally, they imply that $q^*([\nu+2\ell \zeta])$ is primitive for all $\ell \in \mathbb{Z}$.
In particular, $\Hat \kappa_\pm \mp 4 \Hat \eta_\pm = q^*([\nu \pm 4 \zeta])$
is primitive.
Since $( \Hat \eta_\pm, \cE_i) = \mp 1$ and $(\Hat \kappa_\pm \mp 4 \Hat \eta_\pm, \cE_i)=0$ for all $i$,
this implies that
$\Hat \kappa_\pm \mp 4 \Hat \eta_\pm, \Hat \eta_\pm$ (and hence $\Hat \kappa_\pm,\Hat \eta_\pm$)
induce a primitive embedding $\Z^2 \hookrightarrow H^2(\Hat X;\Z)$.
\end{proof}

\section{Constructing the examples in Theorems \ref{maintheorem} and \ref{theorem10}}
\labell{proof}

We now have all the ingredients that we need to construct the symplectic manifolds
described in Theorems 1 and 2, and are ready to assemble them.
We begin by proving our main theorem.

\begin{proof}[Proof of Theorem \ref{maintheorem}]
Let an integer $k \geq 5$ be given.
Choose a positive integer $\ell$ and $k_1,\dots,k_\ell \in \{5,6,7,8,9\}$ such that 
$\sum_{i=1}^\ell k_i = k$.
Define $C_0 := -k_l$ and $C_j :=k_j + \sum_{i=1}^{j-1} 2k_i$ for all $i \in \{1,\dots,\ell\}$.  Note that
$C_1 = C_0 + k_\ell + k_1= C_\ell - 2k + k_\ell + k_1$ and that $C_j = C_{j-1} + k_{j-1} + k_j$ for all $j > 1$.
Finally,  choose an even integer $N$ such that   $N > 2  k_j (k_j  - 4) > 2(k_j - 4)^2$ for all $j$.

By construction, the polynomial $N - 2(t - C_j + k_j)^2$ is positive
on $[C_{j-1} + 4, C_j -4]$ for each $j \in \{1,\dots,\ell\}$.
Therefore, by Proposition \ref{existsfree2}, there exists a free circle action on a symplectic
manifold $(M'_j,\omega'_j)$ with a proper moment map 
$\Psi_j' \colon M_j' \to (C_{j - 1}  + 4, C_j  - 4)$ so that, for all 
$t \in (C_{j - 1} + 4, C_j  - 4)$, the reduced space $M_j' \modt$ is symplectomorphic to a tame
K3 surface $(X_j', I_{j,t}', \sigma_{j,t}')$; moreover,
\begin{itemize}
\item $\int_{X_j'} (\sigma_{j,t}')^2 = N  -2(t - C_j + k_j)^2$; and
\item $[\sigma_{j,t}'] = \kappa_j' - t \eta_j'$, where $\kappa_j'$,  $\eta_j'$
induce a primitive embedding $\Z^2 \hookrightarrow H^2(X_j';\Z)$.
\end{itemize}

Next, for each $j \in \{1,\dots,\ell\}$, we can apply Proposition~\ref{existsfixed2} (centered at $C_j$ instead of at $0$) with $k=k_j$, $A= N - 2k_j(k_j-4)$, and $B=-2$.
Hence, for each such $j$, there exists a circle action on a
symplectic manifold  $(\widetilde{M}_j,\widetilde{\omega}_j)$ with a proper moment map 
$\widetilde{\Psi}_j:\widetilde{M}_j \rightarrow (C_j -4-\widetilde{\epsilon}, C_j  + 4 +\widetilde{\epsilon})$, where $\widetilde{\epsilon} > 0$,
with the following properties:
The level sets $\widetilde{\Psi}_j^{-1}(C_j  \pm 4)$ 
each contain exactly  $k_j$  fixed points with weights $\pm\{-2,1,1\}$; otherwise the action is locally free. 
The Duistermaat-Heckman function of $\widetilde{M}_j$ is
\begin{center}
$\widetilde{\varphi}_j(t)=\begin{cases} 
N - 2(t - C_j + k_j)^2 & t \in (C_j-4- \widetilde{\epsilon},C_j  - 4] \\ 
N - 2k_j(k_j-4) + \frac{k-4}{2}(t - C_j)^2 & t \in [C_j  - 4, C_j  + 4] \\ 
N - 2(t - C_{j} - k_{j})^2 & t \in [C_j + 4,  C_j  + 4 +\widetilde{\epsilon}). 
\end{cases}$
\end{center}
The reduced space $\widetilde{M}_j\modt$ is diffeomorphic to a generalized K3 surface with $k_j$ isolated $\mathbb{Z}_2$ singularities for all  $t \in (C_j  -4,C_j  + 4)$, and symplectomorphic
to a tame K3 surface $(\Hat{X}_j,\Hat{I}_{j},(\Hat{\sigma}_{j\pm})_{t})$, 
where $(\Hat{\sigma}_{j\pm})_{t} \in \Omega^2(\Hat X_j)$ satisfies
 $[(\Hat{\sigma}_{j\pm})_{t}]=\Hat{\kappa}_{j\pm}-t\Hat{\eta}_{j\pm}$
for all  $t \in  C_j  \pm(4,4+ \tepsilon)$; moreover,
$\Hat{\kappa}_{j\pm}$, $\Hat{\eta}_{j\pm}$ induce a primitive embedding $\mathbb{Z}^2 \hookrightarrow H^2(\widetilde{X};\mathbb{Z})$.
Finally, we may assume that $\widetilde{\epsilon} < 1$.

The  above equations imply that
\begin{gather*}
\int_{\Hat{X}_j}(\Hat{\sigma}_{j-})_{t}^2=\widetilde{\varphi}_j(t)=\int_{X_j'}(\sigma_{j,t}')^2 
\ \ \forall j, \ \forall  t \in (C_j  - 4 - \widetilde{\epsilon} , C_j  - 4 ); \\
\int_{\Hat{X}_j}(\Hat{\sigma}_{j+})_{t}^2=\widetilde{\varphi}_j(t)=\int_{X_{j+1}'}(\sigma_{j+1,t}')^2 
\ \ \forall j \neq \ell, \ \forall  t \in (C_j  + 4, C_j  + 4 + \widetilde{\epsilon});  \\
\int_{\Hat{X}_\ell}(\Hat{\sigma}_{\ell+})_{t}^2=\widetilde{\varphi}_\ell(t)=\int_{X_1'}(\sigma_{1, t- 2k}')^2 
\ \  \forall t \in (C_\ell  + 4, C_\ell +  4 + \widetilde{\epsilon}).
\end{gather*}
Therefore, by carefully repeatedly applying Proposition \ref{freeunique2}, there exist $\epsilon$ and $\widehat{\epsilon}$ with $0<\epsilon<\widehat{\epsilon}< \widetilde{\epsilon}$ and equivariant symplectomorphisms
\begin{multline*}
\widetilde{M}_j \supset \widetilde{\Psi}^{-1}(C_j  - 4  - \widehat{\epsilon},C_j  - 4 -\epsilon) \\
\stackrel{\simeq}{\longrightarrow} (\Psi_j')^{-1}(C_j  - 4 - \widehat{\epsilon}, C_j  - 4 - \epsilon ) \subset M_j'\quad   \forall j;
\end{multline*}
\vspace{-.2in}
\begin{multline*}
\widetilde{M}_j \supset \widetilde{\Psi}_j^{-1}(C_j  + 4+\epsilon,C_j  + 4  +\widehat{\epsilon})\\
\stackrel{\simeq}{\longrightarrow} (\Psi_{j+1}')^{-1}(C_j  + 4 +\epsilon,C_j  + 4 +\widehat{\epsilon}) \subset M_{j+1}' \quad \forall j \neq \ell;
\end{multline*}
\vspace{-.2in}
\begin{multline*}
\widetilde{M}_\ell \supset \widetilde{\Psi}_\ell^{-1}(C_\ell + 4+\epsilon,C_\ell  + 4  +\widehat{\epsilon})\\
\stackrel{\simeq}{\longrightarrow} (\Psi_{1}')^{-1}(C_0 + 4 +  \epsilon, C_0 +  4 +\widehat{\epsilon}) \subset M_{1}'.
\end{multline*}

Using these equivariant symplectomorphisms, 
we can glue together $\widetilde{\Psi}_j^{-1}(C_j - 4 -\widehat{\epsilon},C_j + 4  +\widehat{\epsilon})\subset \widetilde{M}_j$ 
and $(\Psi_j')^{-1}(C_{j-1} + 4 +\epsilon,C_j - 4 -\epsilon) \subset M_j'$ for all $j \in \{1,\dots,\ell\}$
 to construct a circle action on a closed, connected six-dimensional 
symplectic manifold $(M,\omega)$ with exactly $2k$ fixed points,
and a generalized moment map $\Psi:M \rightarrow \mathbb{R}/2k\Z$. 
The action is not Hamiltonian, because there is no fixed point whose weights are all positive (or all negative). 
Hence, the manifold $M$ satisfies all the requirements. 
\end{proof}

Theorem \ref{theorem10} also follows from a careful reading of the proof above; 
take $\ell = 1$, $k_1 = 5$, and $N = 12$.

\end{document}